\documentclass[oneside,reqno,a4paper,11pt]{amsart}
\usepackage{amsmath,amsthm,amssymb,amscd,soul}

\usepackage{graphicx}
\usepackage{hyperref}
\usepackage[mathscr]{euscript}
\usepackage{cite}
\usepackage{color}
\usepackage{ifpdf}
\usepackage{fancyhdr}
\usepackage{multirow}
\usepackage{array}
\usepackage{appendix}
\usepackage{longtable}
\usepackage{tikz}

\def\XXint#1#2#3{{\setbox0=\hbox{$#1{#2#3}{\int}$ }
		\vcenter{\hbox{$#2#3$ }}\kern-.6\wd0}}

\newtheorem{theorem}{Theorem}[section]
\newtheorem{lemma}[theorem]{Lemma}
\newtheorem{proposition}[theorem]{Proposition}

\theoremstyle{definition}
\newtheorem{definition}[theorem]{Definition}

\theoremstyle{remark}
\newtheorem{remark}[theorem]{Remark}

\numberwithin{equation}{section}

\newcommand{\beq}{\begin{equation}}
	\newcommand{\eeq}{\end{equation}}
\newcommand{\ben}{\begin{eqnarray}}
	\newcommand{\een}{\end{eqnarray}}
\newcommand{\beno}{\begin{eqnarray*}}
	\newcommand{\eeno}{\end{eqnarray*}}

\newtheorem{mythm}{Theorem}

\numberwithin{equation}{section}
\subjclass[]{}

\begin{document}
	
	\title{The free boundary for the singular obstacle problem with logarithmic forcing term}
	\author{LILI DU$^{1,2}$}
	\author{YI ZHOU$^2$}
	\thanks{* This work is supported by National Nature Science Foundation of China Grant 12125102.}
	\thanks{$^1$ E-mail: dulili@szu.edu.cn \quad $^2$ E-mail: zhouyimath@163.com}
	
	\maketitle
	
	\begin{center}
		$^1$School of Mathematical Sciences, Shenzhen University,
		
		Shenzhen 518061, P. R. China.
	\end{center}
	
	\begin{center}
		$^2$Department of Mathematics, Sichuan University,
		
		Chengdu 610064, P. R. China.
	\end{center}
	
	\begin{abstract}
		
		In the previous work [Interfaces Free Bound., 19, 351-369, 2017], de Queiroz and Shahgholian investigated the regularity of the solution to the obstacle problem with singular logarithmic forcing term
		\begin{equation*}
			-\Delta u = \log u \, \chi_{\{u>0\}} \quad \text{in} \quad \Omega,
		\end{equation*}
		where $\chi_{\{u>0\}}$ denotes the characteristic function of the set $\{u>0\}$ and $\Omega \subset \mathbb{R}^n$ ($n \geq 2$) is a smooth bounded domain. The solution solves the minimum problem for the following functional,
		\begin{equation*}
			\mathscr{J}(u):=\int_{\Omega}\left(\frac{|\nabla u|^2}{2}-u^+ (\log u-1)\right) \, dx,
		\end{equation*}
		where $u^+=\max{\{0,u\}}$.  In this paper, based on the regularity of the solution, we establish the $C^{1,\alpha}$ regularity of the free boundary $\Omega \cap \partial\{u>0\}$ near the regular points for some $\alpha\in (0,1)$.
		
		The logarithmic forcing term becomes singular near the free boundary $\Omega\cap\partial\{u>0\}$ and lacks the scaling properties, which are  very crucial in studying the regularity of the free boundary. Despite these challenges, we draw inspiration for our overall strategy from the "epiperimetric inequality" method introduced by Weiss in 1999 [Invent. Math., 138, 23-50, 1999]. Central to our approach is the introduction of a new type of energy contraction. This allows us to achieve energy decay, which in turn ensures the uniqueness of the blow-up limit, and subsequently leads to the regularity of the free boundary.

	\end{abstract}
	
	\section{Introduction}

	In the previous paper \cite{qs17}, de Queiroz and Shahgholian studied  the regularity of the solution to the obstacle problem with singular logarithmic forcing term
	\begin{equation}\label{eq1.1}
	-\Delta u = \log u \, \chi_{\{u>0\}} \quad \text{in} \quad \Omega,
\end{equation}
	where $\chi_{\{u>0\}}$ denotes the characteristic function of the set $\{u>0\}$ and $\Omega \subset \mathbb{R}^n$ ($n \geq 2$) is a smooth bounded domain. The solution solves the minimum problem for the following functional,
	\begin{equation}\label{eq1}
		\mathscr{J}(u):=\int_{\Omega}\left(\frac{|\nabla u|^2}{2}-u^+ (\log u-1)\right) \, dx,
	\end{equation}
	where $u^+=\max{\{0,u\}}$. Minimization happens in the class,
	$$\mathcal{K}_{\varphi}:=\Big\{u\in W^{1,2}(\Omega);\  u=\varphi \quad\text{on}\quad\partial \Omega\Big\},$$ 
	for a fixed non-negative $\varphi\in H^1(\Omega)\cap L^{\infty}(\Omega)$.  We denote the free boundary $\mathscr{F}(u):=\Omega\cap\partial\{u>0\}$.
	
 The motivation for studying problems involving singular nonlinearities has roots in many applications. Examples can be found in physics and chemical phenomena, as discussed in \cite{a75, d85}. However, it is worth noting that the logarithmic forcing term in equation \eqref{eq1.1} becomes singular near the free boundary $\mathscr{F}(u)$.
 
  In particular, we remark that the observation in the context of the singular obstacle problem \cite{qs17} which states (see \cite[Page 352]{qs17}):
	\begin{align*}
		``&\textit{ ...clearly this is not sharp and, hence, not sufficient to study}\\
		&\textit{fine analytic and geometric properties of the free boundary...}"
	\end{align*}
	This is the one of motivations to investigate the regularity of this problem.
	
	Consider the classical obstacle problem
	\begin{align}\label{eqc}
			\Delta u=\chi_{\{{u>0}\}} \qquad\text{in} \quad \Omega,
	\end{align}
it is clear that $\Delta u$ will exhibit a jump at the free boundary $\Omega\cap\partial\{u>0\}$ from that
\begin{align*}
			\text{ the forcing term }=\begin{split}
		\left\{
		\begin{array}{lr}
			0               & \quad\  \text{on}\quad \partial\{u>0\},\\
			1                  &\text{in}\   \  \quad \{u>0\},
		\end{array}
		\right.
	\end{split}
\end{align*}
a finer result on
the regularity of the solutions of \eqref{eqc}  ($C_{\text{loc}}^{1,1}(\Omega)$) was obtained Frehse \cite{f72}. Particularly, if $	\Delta u =f(x)\chi_{\{{u>0}\}}$ and the forcing term $f(x)$ to have a regular potential $\psi(x) \in C^{1,1}(\Omega)$, namely, $
	f(x)=\Delta \psi (x) \qquad\text{in}\quad \Omega$,  then we also can obtain the optimal regularity $u\in C^{1,1}_{\text{loc}}(\Omega)$ (see \cite[Theorem 2.3]{psu}).
Similar cases for the problem
\begin{align*}
	\Delta u =f(u)\chi_{\{{u>0}\}},
\end{align*}
as in (\cite{w99,asu,dz}), the non-degenerate forcing term $f(u)$ was assumed to satisfy $0<|f(0)|\leq C$ (with $C$ being a constant), which can yield the $C^{1,\alpha}$ ($\alpha\in(0,1)$) regularity of solutions.

On another hand, for the obstacle problem with degenerate forcing term,
\begin{align*}
	\Delta u= u^q\chi_{\{{u>0}\}}\qquad\text{in}\quad \Omega, \qquad\text{for}\quad  q\in(0,1),
\end{align*}
  it gives that $\Delta u$ will not jump on the free boundary $\mathscr{F}(u)$, and the regularity of the solution can reach $C^{[\kappa],\kappa-[\kappa]}_{\text{loc}}(\Omega)$, $\kappa=\frac{2}{1-q}>2$ and $[\kappa]:=\max \{n\in\mathbb{Z};n\leq \kappa\}$ 
    (see  \cite{fsw21} or \cite{f83}).

However, in our case, noting that on the right-hand side of equation \eqref{eq1.1} the logarithmic function is highly singular near 0 ($\log u\to -\infty$ as $u\to 0+$), which implies that a singularity of $\Delta u$ may occur near the free boundary $\mathscr{F}(u)$ and this singularity gives that demonstrating the optimal regularity of solutions becomes a quite delicate problem. de Queiroz and Shahgholian in \cite{qs17} employed very elegant variational approach, derived that any minimizer $u$ of $\mathscr{J}(u)$ \eqref{eq1} belongs to the class $C^{1,\alpha}_{\text{loc}}(\Omega)$, for any $\alpha \in (0,1)$  through growth estimates \cite[Lemma 3.10, Theorem 3.11]{qs17} and non-degeneracy \cite[Lemma 3.14]{qs17} (see Theorem \ref{lem1.1.}). In fact, the optimal regularity of minimizer is $C^{1,\log }_{\text{loc}}(\Omega)$ near the free boundary, namely, 
$$|\nabla u (x)|\leq C d(x)\log \frac{1}{d(x)},\quad\text{for}\quad x\in \Omega'\subset\subset\Omega,$$
where $d(x)=\text{dist} (x,\partial\{u>0\}),\  \text{and constant}\  C>0$.
The analysis involved an approximating variational procedure, wherein the singularity was eliminated, transforming the problem into a regular equation. Moreover, Montenegro and de Queiroz, in \cite{md09}, considered the general case for the singular right-hand side of the equation with homogeneous Dirichlet boundary conditions, in which the existence and regularity properties of a maximal solution (i.e. if $u\geq v$ for any other solution $v$) is obtained.

  Indeed, the importance of the regularity theory for minimizers of free boundary variational problems is well-known at least since the seminal works of Alt--Caffarelli \cite{ac85} for homogeneous problem and Alt--Phillips \cite{ap86} for non-homogeneous problem. It is noteworthy that the regularity of the solution in the classical obstacle problem for free boundary is $C^{1,1}_{\text{loc}}(\Omega)$ (the specific proof details can be found in reference \cite{f18}).	 Fortunately, de Queiroz and Shahgholian in \cite{qs17} showed the following celebrated results on the growth estimate and non-degeneracy of minimizer of the singular problem \eqref{eq1.1}.
	
\renewcommand{\themythm}{\Alph{mythm}}
\begin{mythm}
(The growth estimate and non-degeneracy of minimizer in  \cite{qs17})\label{lem1.1.}
 		Let $u$ be a minimizer of $\mathscr{J}(u)$ in $\mathcal{K}$ and $x^0\in \Omega\cap \partial \{u>0\}$, then we have the inequality 
 		\begin{equation*}
 			C_1\leq \displaystyle\sup_{B_r (x^0)} \frac{u(x^0+rx)}{ r^2|\log r|}\leq C_2,
 		\end{equation*}
 		for some constants $C_1, C_2>0$, provided $B_r(x^0)\subset\subset \Omega$ and $0<r\leq r_0$, for some $r_0<1$ depending only on $\Omega$, $\varphi$ and $n$.
	\end{mythm}
	
	Then the above fact have been utilized to demonstrate that every solution near each free boundary point exhibits supercharacteristic growth
	$$r^2|\log r|,$$
	adding valuable insights to study the regularity of the free boundary $\mathscr{F}(u)$. Furthermore, we understand that such a blow-up limit solution $\left(\displaystyle\frac{u(x^0+rx)}{2r^2|\log r|}\right)$ to the problem \eqref{eq1.1}, converges to the global solution $u_0$ of the classical obstacle problem,
	\begin{align*}
		\Delta u_0=\chi_{\{{u_0>0}\}} \qquad\text{in} \quad \mathbb{R}^n.
	\end{align*}

However, to investigate the free boundary, the problem of "uniqueness of blow-up limit" is a central one in the geometric analysis and free boundary problems (see \cite{asu,w99}), addressing it can lead to energy decay results and subsequently establish the regularity of the free boundary.

	Another noteworthy point is that the minimizers of $\mathscr{J}(u)$ \eqref{eq1} are always non-negative  has been shown in \cite[Lemma 2.1]{qs17}. And then, the non-negativity of minimizers indicates that the minimum problem \eqref{eq1} is equivalent to the following minimum problem (details can be found in \cite[Theorem 1.4]{psu}),
	
	 \begin{equation}\label{eq1.0}
		\mathscr{J}_0(u;\Omega):=\int_{\Omega}\left(\frac{|\nabla u|^2}{2}+u( -\log u+1)\right) dx,
	\end{equation}
	over $$\mathcal{K}_0:=\{u\in W^{1,2}(\Omega); u-\varphi\in W_0^{1,2}(\Omega),\  u\geq 0 \quad\text{in}\quad \Omega\}.$$
	To the end of this paper, for convention, we take 
	\begin{align*}
			F(u)=u( -\log u+1),\quad\text{and}\quad	f(u):=F'(u)=-\log u.
	\end{align*}
	Next, we will introduce the related results on the obstacle problem.
	
\subsection{Related obstacle problems}\

Let us examine the classical obstacle problem,
	\begin{equation}\label{classical}
		\Delta u=\chi _{\{u>0\}}\quad \text{in}\quad \Omega.
	\end{equation}
Caffarelli's dichotomy theorem, as outlined in the celebrated work \cite{c77}, established the mutual exclusion of regular and singular points, specifying that the classification of free boundary points of \eqref{classical} is divided to either regular or singular points, with no overlap between the two categories. 

Moreover, in the classical obstacle problem, Caffarelli's result \cite{c98} relies on the so-called improved flatness, based on the boundary Harnack principle (Athanasopoulos--Caffarelli \cite{ac85}), to bootstrap the Lipschitz regularity of the free boundary near the regular points to $C^{1,\beta}$ regularity,
where \begin{align*}
\mathcal{R}_u:=	\{\text{regular points}\}:=\left\{x^0\in \mathscr{F}(u): \text{any blow-up limit} \lim_{r\to 0+}\frac{u(x^0+rx)}{r^2}=\frac{1}{2}\max (x\cdot \nu, 0)^2\right\},
\end{align*}
for some unit vector $\nu\in\mathbb{R}^n$.
 Additionally, a characterization of the singular points was carried out according to their dimensions (refer to reference \cite{c98,f18}), where 
 \begin{align*}
 	\Sigma_u:=\{ \text{singular points}\}:=\left\{x^0\in \mathscr{F}(u): \text{any blow-up limit} \lim_{r\to 0+}\frac{u(x^0+rx)}{r^2}=\frac{1}{2}\langle Ax,x\rangle \right\},
 \end{align*}
 for some nonnegative definite matrix $A\in\mathbb{R}^{n\times n}$ with $\text{tr}(A)=1$. Furthermore, Figalli and Serra provided a more detailed characterization of the set of singular points (for further details, refer to reference \cite{f18,fs19}).

 To tackle more complex models, often involving time-dependent equations \cite{w00} or systems of equations \cite{asu}, Weiss introduced a groundbreaking method known as the "epiperimetric inequality", which works at regular points and provides an alternative to the methods previously introduced by Caffarelli \cite{c77}. The epiperimetric inequality was first introduced by Reifenberg \cite{r}, White \cite{wh}, and Taylor \cite{ta} in the context of minimal surfaces. Weiss, inspired by the findings of Fleming \cite{f62}, particularly regarding the challenge of determining an oriented minimal surface with a predefined oriented boundary, innovatively applied this approach to demonstrate the regularity of the free boundary for the classical obstacle problem \cite{w99}. The epiperimetric inequality is purely variational approach to the regularity of the free boundary and the minimal surfaces, which gives an estimate on the rate of the energy to its blow-up limit. As a consequence, in \cite{w99}, an epiperimetric inequality for the Weiss adjusted boundary energy related to the classical obstacle problem was proved. This result was used to obtain $C^{1,\beta}$ regularity near the regular points of the free boundary. Furthermore, 
 the methodology was expanded to include additional cases \cite{csv,esv}, with the exact statements as follows.

For the vectorial obstacle problem, the system is elegantly stated as
\begin{equation}\label{ve}
	\Delta \mathbf{u}=\frac{\mathbf{u}}{|\mathbf{u}|}\chi_{\{\mathbf{u}>0\}} \qquad \text{in }\quad \Omega,
\end{equation}
where $\mathbf{u}:=(u_1,u_2,...,u_m)$ represents the vector-valued solution, with $m\geq 2$ components. Notably,  although studies have shown that the regularity of the solution is confined to $C^{1,\alpha}_{\text{loc}}(\Omega,\mathbb{R}^m)$ for some $\alpha\in (0,1)$, an epiperimetric inequality for the vectorial Weiss adjusted boundary energy has been derived in the prestigious work \cite{asu}. This inequality allowed us to quantify the convergence of energy as $r\to 0+$ to be of H\"{o}lder type and provided an alternative proof of the  $C^{1,\beta}$-regularity of free boundary. Very recently, the authors investigated the general vectorial non-degenerate obstacle problem in \cite{dz} and obtained the $C^{1,\beta}$-regularity of free boundary via the method of epiperimetric inequality.

On the other hand, obstacle problem with homogeneous term, 
$$\Delta u = f(u)\chi_{\{{u>0}\}}:=u^q\chi_{\{{u>0}\}},\quad q\in(0,1), \qquad\left(\text{for degenerate case}\quad f(0)=0\right),$$ has been firstly investigated by Weiss in 2000 \cite{w00} and the method of epiperimetric inequality in the original work \cite{w99} also employed  to demonstrate that this free boundary possesses $C^{1,\beta}$ regularity. Furthermore,  this degenerate case has been extended to the vectorial case in \cite{afsw22}.

This paper is a sequel to the previous elegant work by de Queiroz--Shahgholian \cite{qs17}. The purpose of this paper is to investigate the regularity of the free boundary in the non-homogeneous obstacle problem with a log-type singular term. Notably, the logarithmic term is non-homogeneous and unbounded near the free boundary, which causes many difficulties in proving the regularity of the free boundary. For instance, we have the scaling order $2r^2|\log r|$, this is not a polynomial order; the polynomial's degree is multiplied by a logarithmic term, introducing significant difficulties to our calculations. Due to the lack of scaling property for Weiss adjusted boundary energy $W(r;u,x^0)$ (see Definition \ref{weissenergy}), namely,
\begin{align*}
	W(r;u,x^0)\neq W(0+;u_r,0), \qquad\text{for}\quad u_r=\frac{u(x^0+rx)}{2r^2|\log r|},
\end{align*}
 many additional terms are encountered when differentiating a Weiss adjusted boundary energy (in subsection we will delve into the details of this energy). In fact, de Queiroz and Shahgholian have proposed in the abstract of \cite{qs17} that:
 \begin{align*}
 	``&\textit{..., the logarithmic forcing term does not have scaling properties,}\\
 	&\textit{ which are very important in the study of free boundary theory...}"
 \end{align*}

 Moreover, it's hard to show the uniqueness of blow-up limit according to the epiperimetric inequality of the standard Weiss adjusted boundary energy in this situation. However, to overcome this difficulties, we firstly introduce a new Weiss adjusted boundary energy with the variable parameter $\alpha(r)$ (please see \eqref{alp}) such that 
 \begin{align}\label{k}
 	W(r;u,x^0)-Q(r;u,x^0)= K(r;u,x^0)\geq 0, 
 \end{align}
where $Q(r;u,x^0)$ is integrable related to $r$, this also gives $W(0+;u,x^0)\geq0$.
 Secondly, we will show the contraction (epiperimetric inequality) of a new energy with an additional good term (see \eqref{Q}) (integrable in $(0,r)$).  Fortunately, the additional term in the new energy is good in the sense that it converges to zero as $r\to 0+$ (please see \eqref{epip} and \eqref{Q}). This helps us achieve the energy decay then uniqueness of blow-up limit can be obtained, therefore we can yield the regularity of the free boundary as in the pioneering work \cite{w99}.

In summary, we can conclude the results we mentioned above in the table below.

\begin{table}[h]
	\caption{Conclusions for obstacle problems}\label{table1}
	\centering
	\begin{tabular}{|m{4em}<{\centering}|m{10.5em}<{\centering}|m{12em}<{\centering}|m{10.5em}<{\centering}|m{2.5em}<{\centering}|m{16em}<{\centering}|m{2.5em}<{\centering}|}
		\hline
		Problems&Classical obstacle problem (non-degenerate case)
		&Homogeneous obstacle problem (degenerate case)&Non-homogeneous obstacle problem (singular case)\\
		\hline
		Euler-Lagrange equation&  $\Delta u=\chi_{\{{u>0}\}}$& $\Delta u=u^q\chi_{\{{u>0}\}} $, $(q\in(0,1))$& $\Delta u=-\log u \chi_{\{{u>0}\}}$ \\
		\hline
		Regularity of solu. & $C^{1,1}_{\text{loc}}(\Omega)$ \cite{f72}& $C^{[\kappa],\kappa-[\kappa]}_{\text{loc}}(\Omega),$($\kappa=\frac{2}{1-q}>2$) \cite{f83}& $C^{1,\log } _{\text{loc}}(\Omega)$  \cite{qs17}\\
		\hline
		Scaling sequence& $\displaystyle\frac{u(x^0+rx)}{r^2}$ & $\displaystyle\frac{u(x^0+rx)}{r^{\kappa}}$ & $\displaystyle\frac{u(x^0+rx)}{r^2|\log r|}$\\
		\hline
		Regularity of FB& $C^{1,\beta}$ for some $\beta\in(0,1)$\cite{w99}&$C^{1,\beta}$ for some $\beta\in(0,1)$\cite{w00}&{\color{blue}{$C^{1,\beta}$ for some $\beta\in(0,1)$
		[Theorem \ref{regularity}, main result in this paper]}}\\
		\hline
	\end{tabular}
\end{table}

		\subsection{Basic definitions and notations} \
		
In this subsection, we will give some definitions and notations in this paper. Firstly, for the sake of clarity, we provide a list of notations used in our paper.

$\bullet$ $B_1:=B_1(0)$, $B_r:=B_r(0)$, and $B_r^{\pm}(x^0):=\{x=(x_1,...,x_n)\in B_r(x^0): x_n\gtrless x^0_n\}$;

$\bullet$ denote $\mathbf{n}$ as the topological outward normal of the boundary of a given set, and $\nabla_{\theta}f:=\nabla f- (\nabla f\cdot \mathbf{n})\mathbf{n}$ as the surface derivative of a given function $f$;

$\bullet$ the notation $o(t)$ represents an infinitesimal term of higher order than $t$, i.e., $\displaystyle\lim_{t\to 0}\displaystyle\frac{o(t)}{t}=0$.

Secondly, in \cite{w99}, Weiss first introduced the influential Weiss adjusted boundary energy $W(r;u,x^0)$ (or $W(r;u)$ when $x^0=0$ and $W(u)$ when $x^0=0$ and $r=1$) to deal with the classical obstacle problem, which has become a focal point of scholarly exploration due to the fact that it not only provides the homogeneity of the blow-up limits but also allows for the definition of free boundary points using energy density. Recalling the pioneering work \cite{w99}, if $u$ is a minimizer of
\begin{equation}\label{c}
	\mathscr{J}(u;B_r(x^0)):=\int_{B_r(x^0)}|\nabla u|^2 +u dx,
\end{equation}
over the class $$\{u\in W^{1,2}(\Omega); u-\varphi \in W^{1,2}_0(\Omega), u\geq 0 \quad\text{in}\quad \Omega\},$$ for $\varphi\in W^{1,2}(\Omega)$ satisfying $\varphi\geq0$, then so-called Weiss adjusted boundary energy  
\begin{align}\label{cm}
	W(r;u,x^0):=&r^{-n-2}\mathscr{J}(u;B_r(x^0))-2r^{-n-3}\int_{\partial B_r(x^0)} u^2 d\mathcal{H}^{n-1}\nonumber\\
	=&\mathscr{J}(u_r;B_1)-2\int_{ \partial B_1} u_r^2 d\mathcal{H}^{n-1},
\end{align}
 is monotonically increasing with respect to $r$ (see \cite[Theorem 2]{w99}). Here $u_r:=\displaystyle\frac{u(x^0+rx)}{r^2}$, $x^0\in \mathscr{F}(u) $ and $0<r<\text{dist}(x^0,\partial \Omega)$.  A consequence of this monotonicity is that every blow-up limit is a 2-homogeneous globally defined minimizer to \eqref{c}. Notice that the energy \eqref{cm} implies that this functional has scaling property, namely,
 \begin{align}\label{se}
 	r^{-n-2} \mathscr{J}(u;B_r(x^0))=\mathscr{J}(u_r;B_1).
 \end{align}

In our article, we consider the following functional as in \eqref{eq1.0},

\begin{equation*}
	\mathscr{J}_0(u;B_r(x^0))=\int_{B_r(x^0)} \frac{1}{2}|\nabla u|^2+F(u)dx,
\end{equation*}
where $	F(u)=u( -\log u+1)$. Notice that the function $F(u)$ involves a logarithmic term, which does not possess the  scaling property, in other words,
\begin{align}\label{noscaling}
	\frac{1}{r^{n+2}|2\log r|^2}\mathscr{J}_0(u;B_r(x^0))&=\int_{ B_1}\frac{1}{2}|\nabla u_r|^2+\frac{1}{|2\log r|} (-\log \left(u_r r^2|2\log r|\right)+1)\nonumber\\
	&\neq \mathscr{J}_0(u_r;B_1), 
\end{align}
where $u_r=\displaystyle\frac{u(x^0+rx)}{r^2|2\log r|}$, moreover, the term ($\displaystyle\frac{1}{|2\log r|} (-\log \left(u_r r^2|2\log r|\right)+1)$) is not exclusively dependent on $u_r$. This introduces complexities when proving the monotonicity of $W(r;u,x^0)$ respect to $r$.  To address this, we propose a new Weiss adjusted boundary energy with a variable parameter $\alpha(r)$ which better helps us construct perfect square terms (see \eqref{W'}) after differentiating $W(r;u,x^0)$. Specifically, noting that $|\log r|=-\log r$ for $0<r\ll 1$, in the situation when we choose rescaling $$u_r(x):=\frac{u(x^0+rx)}{\mu(r)}:=\frac{u(x^0+rx)}{1-2r^2\log r},$$ to construct a complete square term (as $K(r;u,x^0)$ in \eqref{k}) for computing the derivatives $\displaystyle\frac{d}{dr}W(r;u,x^0)$ we will take
\begin{align*}
	\alpha(r)=\frac{2\mu(r)}{r\mu'(r)}=1-\frac{1}{2\log r}.
\end{align*}

\begin{remark}
	If we choose the scaling
	\begin{align*}
		\psi(r):=-2r^2\log r \qquad(0<r\ll 1),
	\end{align*}
	in stead of $\mu(r)$, it follows that
	\begin{align*}
		\alpha(r)=\frac{2\psi(r)}{r\psi'(r)}=1-\frac{1}{2\log r +1}.
	\end{align*} 
Under this scaling, our subsequent proofs can still hold true.

\end{remark}

Therefore, we provide following definitions that will be relevant for our work.
\begin{definition}
	Let $u$ be a solution of the problem \eqref{eq1.1} in $B_{r_0}(x^0)\subset \Omega$, then, for any $x^0\in \mathscr{F}(u)$, and $x\in B_{r_0}(x^0)$, we can define the \textit{blow-up sequence}
	$$u_r(x):=\frac{u(x^0+rx)}{\mu(r)}=\frac{u(x^0+rx)}{r^2(1-2\log r)},$$
	and then if $u_r(x)$ weakly converges to $u_0(x)$ in $W^{1,2}(\Omega)$, we say $u_0(x)$ is a \textit{blow-up limit} at $x^0$.
\end{definition}

\begin{definition} {\it(Weiss adjusted boundary energy) }\label{weissenergy}
	Let $u$ be a solution of the obstacle problem \eqref{eq1.1} in $B_{r_0}(x^0)$. Then one can define the following \textit{Weiss adjusted boundary energy} with variable parameter $\alpha(r)$
	\begin{equation}\label{weiss energy}
		\begin{aligned}
			W(r;u, x^0):=&\frac{\alpha(r)}{r^{n+2}(1-2\log r)^{2}}\mathscr{J}_0(u;B_r(x^0))\\
			&-\frac{1}{r^{n+3}(1-2\log r)^{2}}\int_{\partial B_r(x^0)}u^2 d\mathcal{H}^{n-1},
		\end{aligned}
	\end{equation}
	for $0<r\leq r_0<1$, where $$\mathscr{J}_0(u;B_r(x^0))=\int_{B_r(x^0)} \frac{1}{2}|\nabla u|^2+F(u)dx,$$
	\begin{align}\label{alp}
		\alpha(r)=1-\displaystyle\frac{1}{2\log r},
	\end{align}
	and $B_r(x^0)$ denotes an open ball with radius $r$ in $\mathbb{R}^n$ centered at $x^0$, and $\mathcal{H}^{n-1}$ denotes ${(n-1)}-$dimensional Hausdorff-measure.
\end{definition}

The idea regarding the Weiss adjusted boundary energy with a variable parameter is borrowed from those used in the two-phase obstacle problem discussed in \cite{ks21}. There is a good observation that we also can obtain that the blow-up limit is a homogeneous function of degree 2.

	Motivated by the classical obstacle problem, we split the free boundary $\mathscr{F}(u)$ of minimizer $u$ in regular and singular part according to the forms of the blow-up limits. Therefore, we firstly give the definition of the set of half-space solutions.
	\begin{definition} {\it(Half-space solutions)}
		The set of half space solutions is given by
	\begin{equation}\label{eq3.1}
	\mathbb{H}:=\left\{ h_{\nu}=\frac{1}{2}\max(x\cdot\nu,0)^2: \nu \ \text{unit\ vector\ of }\ \mathbb{R}^n\right\}.
\end{equation}
\end{definition}	
	
According to the definition of the set of half-space solutions, we given the following definition of regular points,
\begin{align*}
	\mathcal{R}_u :=\{x\in \mathscr{F} (u): \text{any blow-up limit at} \ x \ \text{is of the form } h_{\nu}\in\mathbb{H}\}.
\end{align*}

Due to Caffarelli's dochotomy theorem to the classical obstacle problem in the groundbreaking work \cite{c77} indicates that among all free boundary points, except for the case where the blow-up limits are half-space, the blow-up limits of the remaining free boundary points are in the form of quadratic polynomials ($\langle Ax,x\rangle $ for some nonnegative definite matrix $A\in\mathbb{R}^{n\times n}$ ). However, we are currently unable to achieve dichotomy for log-type obstacle problem, so we define free boundary points with blow-up limits other than half-spaces as singular points, namely,
\begin{align*}
	\Sigma_u :=\{x\in \mathscr{F} (u): \text{at least one blow-up limit at} \ x \ \text{is not of the form } h_{\nu}\in\mathbb{H}\}.
\end{align*}

\begin{remark}\label{rem1.4}
	Noted that the energy $\displaystyle \lim_{r\to 0+} W(r;u,x^0)$ is dimensional constant for any $x^0\in \mathcal{R}_u$ (we refer to this constant as to the \textit{energy density} at the regular points and denoted it by $\displaystyle\frac{\omega_n}{2}$), precisely, we have 
$$\displaystyle \lim_{r\to 0+} W(r;u,x^0)=\frac{\mathcal{H}^{n-1}(\partial B_1)}{8n(n+2)}:=\displaystyle\frac{\omega_n}{2},\qquad \text{for any}\quad x^0\in \mathcal{R}_u.$$ 
\end{remark}
For the calculations here, please refer to Lemma \ref{property} (2) in conjunction with the definition of $\mathcal{R}_u$.
As we mentioned before, the energy $W(r;u,x^0)$ in \eqref{weiss energy} here lacks of a scaling property due to the logarithmic type term. Indeed, let we recall the scaling property \eqref{se} as mentioned the problem \eqref{cm} in \cite{w99},  which satisfies the following equation
\begin{equation*}
	W(r;u)=W(1;u_r):=\mathscr{J}(u_r;B_1)-2\int_{\partial B_1} u_r^2d\mathcal{H}^{n-1},
\end{equation*}
for $W(r;u,x^0)$ defined in \eqref{cm}. This shows that the \textit{balance energy }can be given as 
\begin{equation}\label{M1}
	M_0(v):=\mathscr{J}(v;B_1)-2\int_{\partial B_1} v^2d\mathcal{H}^{n-1},
\end{equation}
it is worth mentioning that this $M_0(v)$ is independent of $r$. Specifically, by taking $v=u_r$ in \eqref{M1}, it implies that $M_0(u_r)=W(1;u_r)$. Consequently, the author established the contraction (i.e. epiperimetric inequality) of the balance energy for $M_0(v)$.  However, in our situation, it does not possess scaling property, namely,
\begin{equation*}
	W(r;u)\neq W(1;u_r),
\end{equation*}
and it suggests that, in contrast to the well-established frameworks of the classical obstacle problem \cite{w99} and the homogeneous case \cite{w00}, we encounter a challenge in identifying a function $M_0(v)$ (similar to the aforementioned functional \eqref{M1}) due to the transformation undergone by the form upon scaling. Specifically, when applying a coordinate transformation to the Weiss adjusted boundary energy, it implies that compared with the classical obstacle problem \cite{w99} and homogeneous case \cite{w00} we cannot directly find the functional $M_0(v)$, since the form has been changed after scaling. Namely, we perform a coordinate transformation for the Weiss adjusted boundary energy
\begin{equation*}
	W(r;u,x^0)=\alpha(r) \int_{ B_1}\frac{1}{2}|\nabla u_r|^2+{G}(r;u_r) dx-\int_{\partial B_1} u_r^2 d\mathcal{H}^{n-1},
\end{equation*}
where ${G}(r;v)$ is given by
\begin{equation*}
	{G}(r;v)=\displaystyle\frac{v}{1-2\log r}\left[-\log \left(vr^2(1-2\log r)\right) +1\right],\quad\text{for any}\quad v\geq 0.
\end{equation*}
To address this issue and cater to our subsequent requirements in establishing the epiperimetric inequality, we introduce a novel function $M(r;v)$. This functional accounts for the additional dependence on the radial parameter $r$ as follows,

\begin{align}\label{M(r;v)}
	M(r;v):=\alpha(r)\int_{B_1}\frac{1}{2}|\nabla v|^2+{G}(r;v)dx-\int_{\partial B_1} v^2 d\mathcal{H}^{n-1},
\end{align}
and 
\begin{align}\label{M_0(v)}
	\displaystyle\lim_{r\to 0+} M(r;v) =M_0(v):=\int_{B_1}\frac{1}{2}|\nabla v|^2+vdx-\int_{\partial B_1} v^2 d\mathcal{H}^{n-1}.
\end{align}
Then we will establish the contraction of energy (via epiperimetric inequality) for $M_0(v)$, which gives the uniqueness of the blow-up limit.

Furthermore, we would like to emphasize that $M_0(h_{\nu})$ takes for any $h_{\nu}\in\mathbb{H}$ the uniform value for any $\nu \in \partial B_1$, namely

	\begin{align}\label{alp_n}
		M_0(h_{\nu})=M_0\Big(\frac{1}{2}\max(x\cdot\nu,0)^2\Big)=\frac{1}{4}\int_{B_1}\max(x\cdot\nu,0)^2 dx =\frac{1}{8n}\int_{ B_1}|x|^2dx =\frac{\mathcal{H}^{n-1}(\partial B_1)}{8n(n+2)}:=	\frac{\omega_n}{2}.
	\end{align}
It means that for any $x^0\in \mathcal{R}_u$, based on Remark \ref{rem1.4}, we can deduce that
\begin{align*}
	\lim_{r\to 0+} W(r;u,x^0)=M_0(h_{\nu})=\frac{\omega_n}{2}.
\end{align*} 

	Consequently, based on above result and \cite[Corollary 2]{w99}, a fresh perspective is offered concerning the set of regular free boundary points (denoted as $\mathcal{R}_u$), namely, we can use density to determine the form of the blowup limit as proposition \ref{regular}. This set is defined as the collection of free boundary points where at least one blow-up limit coincides with a half-plane solution. The key rationale lies in the observation that half-plane solutions $h_{\nu}$ exhibit a lower energy density level $M_0(u)$ than any other homogeneous solution $u$ of degree 2 (see Theorem \ref{isolated}).

	\subsection{Main results and plan of this paper}\
	
		In this paper, the primary outcome regarding the regularity of the free boundary $\mathscr{F}({u})$ is outlined in the following theorem. 
	
	\begin{theorem}{(Regularity of free boundary)}\label{regularity}
		The free boundary $\mathscr{F}({u})$ is in an open neighbourhood of the regular points set $\mathcal{R}_{{u}}$ locally a $C^{1,\beta}$-surface. Here $\beta:=\frac{2\eta+n\eta}{2+n\eta}\in(0,1)$ for some $\eta\in(0,1)$.
	\end{theorem}

	\begin{remark}
		In fact, the parameter $\eta$ is a contractive parameter determined by the "epiperimetric inequality" stated in following Lemma \ref{epip}.
	\end{remark}

	One of the key elements in the proof of Theorem \ref{regularity} is the uniqueness of the blow-up limits near the regular points. The epiperimetric inequality is a fundamental tool in establishing the energy decay to obtain the uniqueness of the blow-up limit which ensures that our analysis is robust and leads to the $C^{1,\beta}$ regularity of the regular part of the free boundary.

	\begin{remark}
		Recently, Fotouhi and Khademloo \cite{fk} conducted a comprehensive investigation into the existence, optimal regularity, growth estimates, and non-degeneracy of minimizers at free boundary points for an obstacle problem with more singular forcing term $f(u)=u^q\log u$ for $u\geq 0$ and $q\in(-1,0)$. Notably, the novelty lies in the intensified singularity on the right-hand side as $q$ approaches $-1$ when considering the behavior near the free boundary. Therefore, we will investigate the regularity of the free boundary in our future work.
	\end{remark}
	
	Next, we present our fundamental tool epiperimetric inequality. Unlike existing literature, we introduce a new energy incorporating an additional term (see \eqref{Q}) (integrable in $(0,r)$) with the variable parameter $\alpha(r)$, offering a fresh perspective on the epiperimetric inequality. This approach is expected to be flexible enough to apply to a wide variety of problems.

		\begin{lemma}{ (Epiperimetric inequality)}\label{epip}
		There exist $ \eta \in (0,1)$ and $ \delta>0 $, such that if $\mathcal{C}$ is a non-negative homogeneous global function of degree $2$ satisfying $$\|\mathcal{C}-h_{\nu}\|_{W^{1,2}(B_1)}\leq \delta,\quad \text{for some}\quad h_{\nu}\in\mathbb{H},$$ then there exists a function $v\in W^{1,2}(B_1)$ with $v=\mathcal{C}$ on $\partial B_1$ satisfying
		\begin{equation}\label{6.1}
			M(s;v)-M_0(h_{\nu})-\int_{0}^{s}\overline{Q}(r;u,x^0)dr\leq (1-\eta) \left(M(s;\mathcal{C})- M_0(h_{\nu}) -\int_{0}^{s}\overline{Q}(r;u,x^0)dr\right),
		\end{equation}
			for any  $s\in(0, s_0(\delta)]$, where $s_0(\delta)=o\left(e^{-\delta^{-\frac{2}{\gamma}}}\right)$, moreover, $M(s;v)$ and $M(s;\mathcal{C})$ following the formulation given in \eqref{M(r;v)}, along with $M_0(h_{\nu})$  from the definition \eqref{M_0(v)}. Additionally, we introduce a modified function $\overline{Q}(r;u,x^0)$, which is obtained by subtracting a specific term from $Q(r;u,x^0)$, as shown in following equation,
		
		\begin{equation}\label{Q}
			\overline{Q}(r;u,x^0)=Q(r;u,x^0)-\frac{1}{r(1-2\log r)^{1+\gamma}},\qquad \text{for any}\ \gamma\in(0,1). 
		\end{equation}
		The function $Q(r;u,x^0)$ itself, as established in Lemma \ref{monotonicity}, is integrable over the interval $(0,r_0)$.
	\end{lemma}
	
	\begin{remark}
		In this epiperimetric inequality, to overcome the difficulties posed by the fact that $\alpha(r)$ is a variable and there is an additional term involving the integral of the integrable function $\overline{Q}(r;u,x^0)$, we set $s_0(\delta)=o\left(e^{-\delta^{-\frac{2}{\gamma}}}\right)$. Specifically, this implies that $\left(-\log s_0(\delta)\right)^{-1}=o(\delta^{\frac{2}{\gamma}})$, which allows our additional term to approach zero as $\delta\to 0$. This adjustment ensures that as $\delta\to 0$, the impact of the variable $\alpha(r)$ and the integral of $\overline{Q}(r;u,x^0)$ becomes negligible, facilitating the analysis and simplification of the inequality. The detailed application of this can be found in Section 3, where the epiperimetric inequality is proven.
	\end{remark}

	Epiperimetric inequality has been used in both minimal surfaces and obstacle-type free boundary problems, they can be divided into two categories according to the methods of proofs:
	
	\textit{The direct epiperimetric inequalities} was initially presented by Reifenberg \cite{r} and White \cite{wh} within the realm of minimal surfaces in the 60s. This idea was later expanded to address the free boundary context by the authors in \cite{csv18,sv,sv21}. The direct method relies on the explicit construction of the competitor $v$ and is particularly suited for deriving decay estimates around singular points.
	
	\textit{The epiperimetric inequalities by contradiction} are based on the compactness of the sequence of linearizations, as discussed in references such as \cite{w99,ta}. However, the contradiction arguments available in the literature are typically applicable only to regular points or to singular points under specific and restrictive conditions, which might also apply to the integrable case in our framework. Therefore, we will use contradiction approach in the proof of the epiperimetric inequality presented in Section 3.
	
	Finally, we show that Lemma \ref{epip} yields the desired uniqueness of the blow-up.

	\begin{proposition}{(Energy decay and uniqueness of blow-up limit)}\label{uniqueness}
		Let $x^0\in \mathscr{F}(u)$ and suppose that the epiperimetric inequality \eqref{6.1} holds true with $\eta\in(0,1)$ and $r_0 \in (0,1)$, for each $\mathcal{C}_r$ as follows,
		$$\mathcal{C}_r(x):=|x|^2 u_r\left(\frac{x}{|x|}\right)=\frac{|x|^2}{r^2(1-2\log r)} u\left(x^0+\frac{r}{|x|}x\right), \quad \text{for any}\quad 0<r\leq r_0<1.$$
		Assume that $u_0$ denotes an arbitrary blow-up limit of $u$ at $x^0$.
		Then there holds
		\begin{equation}\label{eq4.3}
			\left|\overline{W}(r;u,x^0)-\overline{W}(0+;u,x^0) \right|\leq \left|\overline{W}(r_0;u,x^0)-\overline{W}(0+;u,x^0) \right| \left(\frac{r}{r_0}\right)^{\frac{(n+2)\eta}{1-\eta}}, 
		\end{equation}
		for $r\in(0,r_0)$, where $\overline{W}(r;u,x^0):=W(r;u,x^0)-\int_{0}^{r}\overline{Q}(s;u,x^0)ds$,  and there is a constant $C(n)>0$ such that
		\begin{equation}\label{eq4.4}
			\int_{\partial B_1}\left|u_r(x)-u_0(x)\right|d\mathcal{H}^{n-1}\leq C(n) \left|\overline{W}(r_0;u,x^0)-\overline{W}(0+;u,x^0) \right|^{\frac{1}{2}}\left(\frac{r}{r_0}\right)^{\frac{(n+2)\eta}{2(1-\eta)}},
		\end{equation}
		for $r\in(0,\frac{r_0}{2})$, and $u_0$ is the unique blow-up limit of $u$ at $x^0$.
	\end{proposition}
	
	What we need to note here is that, when $\mathcal{C}_r(x)$ exhibits homogeneity, $\log \mathcal{C}_r(x)$ does not share this property. This is unlike the case discussed by Weiss in \cite{w99,w00,asu}, and it also contradicts the convexity of $F(u)$ as discussed by the authors in \cite{dz}. Therefore, we need to handle the terms with logarithmic carefully. Since our ultimate aim is to achieve uniqueness in the context of blow-up limit, we make some adjustments to the energy $W(0+;u,x^0)$. First, we will prove the decay rate of 
	$$\overline{W}(r;u,x^0):=W(r;u,x^0)-\int_{0}^{r}\overline{Q}(s;u,x^0)ds,$$
	where $	\overline{Q}(r;u,x^0)=Q(r;u,x^0)-\displaystyle\frac{1}{r(1-2\log r)^{1+\gamma}},$ for any $\gamma\in(0,1)$ as in \eqref{Q}. For more information about $\overline{Q}(r;u,x^0)$, we refer to Remark \ref{rem3.2} and Remark \ref{rem3.3}.

	\begin{remark}
		$\overline{W}(r;u,x^0)$ is indeed monotonically increasing related to $r$, thus the difference $\left|\overline{W}(r;u,x^0)-\overline{W}(0+;u,x^0) \right|$ can be written as $\overline{W}(r;u,x^0)-\overline{W}(0+;u,x^0) $ for any $r>0$.
	\end{remark}
		In fact, monotonicity formula $\displaystyle\frac{d}{dr}W(r;u,x^0)-Q(r;u,x^0)\geq 0$ implies that $\displaystyle\frac{d}{dr}\overline{W}(r;u,x^0)=\displaystyle\frac{d}{dr}W(r;u,x^0)-\overline{Q}(r;u,x^0)\geq 0$, which shows that $\overline{W}(r;u,x^0)$ is monotonically increasing and
	$$\overline{W}(r;u,x^0)\geq \overline{W}(0+;u,x^0).$$

	Our research approach is greatly influenced by the renowned work \cite{w99}, particularly the establishment of the epiperimetric inequality. The subsequent sections of this paper will be organized as follows.

	In Section 2, we will discuss the Weiss-type monotonicity formula with variable parameter, which helps us understand the homogeneity of blow-up solutions. Subsequently, we will employ a proof by contradiction to introduce the important tool "epiperimetric inequality" (Section 3), leading to the uniqueness of blow-up solutions in Section 4. Finally, in Section 5, we will establish the H\"older continuity of the normal vectors $\nu$ at free boundary points, thereby proving the main theorem \ref{regularity}.

	\vspace{25pt}
	\section{Weiss-type monotonicity formula with variable parameter}

	In this section, we will introduce Weiss-type monotonicity formula, which was introduced firstly in the elegant work \cite{w98} by Weiss to deal with the classical obstacle problem, and then it has a wide range of applications to the obstacle problem such as in \cite{w99} and  \cite{w01}. In addition, the Weiss-type monotonicity formula for Bernoulli-type free boundary problems also plays an important role, as shown in \cite{vw12}. One of key ingredients in this paper is that the monotonicity formula gives the existence of $\displaystyle\lim_{r\to 0+} W(r;u,x^0)$. In fact, the energy $W(r;u,x^0)$ with variable parameter $\alpha(r)$ is also not monotone with respect to $r$, however the variable parameter $\alpha(r)$ helps us that $\displaystyle\frac{dW(r;u,x^0)}{dr}$ can be divided into the following two terms, namely,
	
	\begin{equation*}
		 \frac{dW(r;u,x^0)}{dr}= K(r;u,x^0)+Q(r;u,x^0),
	\end{equation*} 
	where $K(r;u,x^0)$ is a non-negative term (perfect square term) and $Q(r;u,x^0)$ is integrable term.

	\begin{lemma} {(Weiss-type monotonicity formula) }\label{monotonicity}
		Let $u$ be a solution of \eqref{eq1.1} in $B_{r_0}(x^0)$, then there holds
		\begin{equation*}
			\frac{dW(r;u,x^0)}{dr}=\frac{\alpha (r)}{r}\int_{\partial B_1} \left(\nabla u_r\cdot x -\frac{2}{\alpha(r)}u_r\right)^2d\mathcal{H}^{n-1}+Q(r;u,x^0),
		\end{equation*}
		where
		$\alpha(r)=1-\displaystyle\frac{1}{2\log r}$, and
		\begin{equation*}
			\begin{aligned}
				Q(r;u,x^0)=&\frac{1}{2r(\log r)^{2}}\int_{B_1} \frac{1}{2}|\nabla u_r|^2 +\frac{u_r}{1-2\log r}\left[-\log \left(u_r r^2(1-2\log r)\right)+1\right] dx\\ &+\left(1-\frac{1}{2\log r}\right)\int_{B_1}\frac{2 u_r}{r(1-2\log r)^2}\left[-\log \left(u_r(1-2\log r)\right)+1\right] dx.
			\end{aligned}
		\end{equation*}
		
	\end{lemma}
		 \begin{remark}
		We notice that the variable parameter $\alpha(r)$ is introduced here,
		\begin{equation*}
			\alpha(r)=1-\displaystyle\frac{1}{2\log r}\to  1\quad \text{as} \quad r\to 0+,
		\end{equation*}
		 this finding tells us that if
		\begin{equation*}
			\displaystyle\lim_{r\to 0+} \left(\nabla u_r(x)\cdot x-\frac{2}{\alpha(r)} u_r(x)\right) =0\quad\text{a.e. in}\quad B_{\frac{1}{r}}(0),
		\end{equation*}
		it also implies that blow-up limit $u_0$ is a homogeneous function of degree 2, namely,
		\begin{equation*}
				\displaystyle\nabla u_0(x)\cdot x-2 u_0(x) =0\quad\text{a.e. in}\quad \mathbb{R}^n.
		\end{equation*}
	\end{remark}
 
\begin{remark}
	However, the variable parameter $\alpha(r)$ assists us to construct the perfect square terms (see \eqref{W'}) directly. Although it generates additional terms, fortunately, a careful analysis gives that these extra terms $Q(r;u,x^0)$ are integrable near $r\to0$. Hence, we can obtain the homogeneity of the blowup limit $u_0$ in this situation.
\end{remark}

	\begin{proof}
		
		In this proof, we proceed through direct computation to find the derivative of the energy $W(r;u,x^0)$. To accomplish this, we first perform a coordinate transformation, which allows us to simplify the expression and facilitate the calculation of the derivative,
		\begin{align*}
			W(r;u,x^0)=\alpha(r) \int_{ B_1}\frac{1}{2}|\nabla u_r|^2+{G}(r;u_r) dx-\int_{\partial B_1} u_r^2 d\mathcal{H}^{n-1}.
		\end{align*}
	Here,
		\begin{align*}
			{G}(r;v)=\displaystyle\frac{v}{1-2\log r}\left[-\log \left(vr^2(1-2\log r)\right) +1\right],   \quad\text{for}\quad v\geq0.
		\end{align*}

		Next, we proceed directly to find the derivative, which yields
		
       \begin{align}\label{W'1}
			\frac{dW(r;u,x^0)}{dr}\nonumber
			=&\alpha'(r)\int_{ B_1}\frac{1}{2}|\nabla u_r|^2+{G}(r;u_r) \ dx+\alpha(r)\int_{ B_1}\nabla u_r\frac{d}{dr}\left(\nabla u_r\right)dx\\
			&+\alpha(r)\int_{ B_1}\frac{d}{dr}{G}(r;u_r)dx-\int_{\partial B_1}2u_r\frac{du_r}{dr}d\mathcal{H}^{n-1}.
		\end{align}

		Note that,
		\begin{itemize}
			\item $\alpha'(r)=\displaystyle\frac{1}{2r(\log r)^2}$;
			\\[3pt]
			\item A direct calculation gives that \begin{align*}
					\displaystyle\frac{du_r(x)}{dr}&=\frac{d}{dr}\left(\frac{u(rx+x^0)}{r^2(1-2\log r)}\right)\\
					&=\frac{\nabla u_r(x)\cdot x}{r}-u_r(x)\frac{\left(r^2(1-2\log r)\right)'}{r^2(1-2\log r)}\\
			     &=\frac{1}{r}\left[\nabla u_r\cdot x-\frac{2}{\alpha(r)}u_r(x)\right];
			     \end{align*}
	 
			\item 
			For ${G}(r;u_r) \not\equiv0$, one has
			$$\displaystyle\frac{d}{dr}{G}(r;u_r)=\frac{2u_r}{r(1-2\log r)^2}\left[-\log\left(u_r(1-2\log r)\right)+1\right]+\Delta u_r\frac{d u_r}{dr}.$$
		
			Indeed, for ${G}(r;u_r)\not\equiv 0$, it can be yielded that
			$$\displaystyle\frac{d}{dr}{G}(r;u_r)=G_1(r;u_r)+G_2(r;u_r)\frac{du_r}{dr},$$
			where
			\begin{equation*}
					G_1(r;u_r):=\frac{\partial}{\partial r}{G}(r;u_r)=\frac{2u_r}{r(1-2\log r)^2}\left[-\log\left(u_r(1-2\log r)\right)+1\right],
			\end{equation*}
			and
			\begin{equation*}
				G_2(r;u_r):=\displaystyle\frac{\partial}{\partial u_r}{G}(r;u_r)
				=\frac{1}{1-2\log r}\left[-\log\left(u_rr^2(1-2\log r)\right)\right].
			\end{equation*}
			Notice that, in $\{u_r>0\}$,
			\begin{equation}\label{u_r}
				\Delta u_r=\frac{\Delta u(x^0+rx)}{1-2\log r}=\frac{-\log u(x^0+rx)}{1-2\log r}=\frac{ -\log\left(u_rr^2(1-2\log r)\right)}{1-2\log r}=G_2(r;u_r).
			\end{equation}
		
		\end{itemize}
		
		We then substitute the above facts into the equation \eqref{W'1} for $\displaystyle\frac{dW(r;u,x^0)}{dr}$ to obtain,

			\begin{align}\label{W'}
				\frac{dW(r;u,x^0)}{dr}=&\alpha(r)\int_{ B_1}-\Delta u_r \frac{d u_r}{dr}dx+\alpha(r) \int_{\partial B_1} \nabla u_r \cdot x \frac{d u_r}{dr} d\mathcal{H}^{n-1}+\alpha(r)\int_{ B_1}G_1(r;u_r)dx\nonumber\\
				&+\alpha(r)\int_{ B_1}\Delta u_r\frac{d u_r}{dr}dx+\alpha'(r)\int_{ B_1}\frac{1}{2}|\nabla u_r|^2+{G}(r;u_r)dx-2\int_{\partial B_1} u_r \frac{d u_r}{dr}d\mathcal{H}^{n-1}\nonumber\\
				=&\int_{\partial B_1}\frac{d u_r}{dr}\left(\alpha (r) \nabla u_r\cdot x -2 u_r\right)\mathcal{H}^{n-1}+\alpha'(r)\int_{ B_1}\frac{1}{2}|\nabla u_r|^2+{G}(r;u_r)dx\nonumber\\
				&+\alpha(r)\int_{ B_1}G_1(r;u_r)dx\nonumber\\
				=&\frac{\alpha (r)}{r}\int_{\partial B_1} \left(\nabla u_r\cdot x -\frac{2}{\alpha(r)}u_r\right)^2 d\mathcal{H}^{n-1}+Q(r;u,x^0),
			\end{align}
		where 

			\begin{align*}
				Q(r;u,x^0)=&\alpha'(r)\int_{ B_1}\frac{1}{2}|\nabla u_r|^2+{G}(r;u_r)dx+\alpha(r)\int_{ B_1}G_1(r;u_r) dx\\
				=&\frac{1}{2r(\log r)^{2}}\int_{B_1} \frac{1}{2}|\nabla u_r|^2 +\frac{u_r}{1-2\log r}\left[-\log \left(u_r r^2(1-2\log r)\right)+1\right] dx\\ &+\left(1-\frac{1}{2\log r}\right)\int_{B_1}\frac{2 u_r}{r(1-2\log r)^2}\left[-\log \left(u_r(1-2\log r)\right)+1\right] dx.
			\end{align*}

	\end{proof}
	
	\begin{remark}[Integrability of $Q(r;u,x^0)$]\label{q(r)}
			We recall Theorem \ref{lem1.1.} concerning the growth estimate and non-degeneracy of the minimizers, which shows that 
			\begin{align*}
				|Q(r;u,x^0)|\leq  C\frac{\log(-\log r)}{r(\log r)^2},
			\end{align*}
			furthermore,
	       \begin{align*}
           \int_{0}^{r}\frac{\log(-\log s)}{s(\log s)^2}ds\to 0\qquad\text{as}\quad r\to 0+.
	       \end{align*}
			Hence, we can conclude that $Q(r;u,x^0)$ is integrable.
	\end{remark}

		Once Lemma \ref{monotonicity} is established, one can obtain the convergence of energy as follows.
		
		\begin{lemma}\label{property} Let $u$ be the minimizer of \eqref{eq1.0} and $x^0\in\mathscr{F}(u)$. Then the following conclusions hold true,

			$\mathrm{(1)}$ The limit $W(0+;u,x^0)$ exists and is finite.
			
			$\mathrm{(2)}$ Let $0<r_m\rightarrow 0$ be a sequence such that the blow-up sequence
			$${u}_m(x)=\frac{u(x^0+r_mx)}{r_m^2(1-2\log r_m)},$$
			converges weakly in $W^{1,2}_{loc}(\mathbb{R}^n)$ to $u_0(x)$. Then ${u}_0(x)$ is a homogeneous function of degree 2.
			Moreover,
			$$W(0+;u,x^0)=\int_{B_1} \frac{1}{2} u_0(x) dx\geq0.$$
			Additional, $W(0+;u,x^0)=0$ implies that $u\equiv0$ in $B_{\delta}(x^0)$ for some $\delta>0$.
			
			$\mathrm{(3)}$ The function $x\longmapsto W(0+;u,x)$ is upper-semicontinuous.
			
		\end{lemma}

		\begin{proof}(1) Since the additional term $Q(r;u,x^0)$ in the monotonicity formula (Lemma \ref{monotonicity}) is integrable with respect to $r$ when $r\in(0, r_0)$, the function $\int^r_0Q(s;u,x^0)ds$ is well-defined. Thus, we obtain that
			$$\frac{d}{dr}\left(W(r;u,x^0)-\int_0^r Q(s;u,x^0)ds\right)\geq 0,$$ 
		it gives that $\displaystyle\lim_{r\to 0+}\left(W(r;u,x^0)-\int_0^r Q(s;u,x^0)ds\right)$ exists. Moreover, thanks to the facts of the non-degeneracy and growth estimate of minimizers in Theorem \ref{lem1.1.}, we infer that the function $W(r;u,x^0)$ has finite right limit $W(0+;u,x^0)$.

			(2) Noticing that the sequence ${u_m}$ is bounded in $W^{1,2}_{loc}(\mathbb{R}^n)$, and the limit $W(0+;u,x^0)$ is finite. By applying the monotonicity formula stated in Lemma \ref{monotonicity}, it can be seen that
		 
				\begin{align*}
					\frac{dW(r;u,x^0)}{dr}=&\frac{\alpha (r)}{r}\int_{\partial B_1} \left(\nabla u_r\cdot x -\frac{2}{\alpha(r)}u_r\right)^2d\mathcal{H}^{n-1}+Q(r;u,x^0)\\
					=&\frac{\alpha(r)}{r^{n+2}(1-2\log r)^2}\int_{\partial B_r(x^0)}\left(\nabla u\cdot \nu-\frac{2}{\alpha(r)}\frac{u}{r}\right)^2 d\mathcal{H}^{n-1}+Q(r;u,x^0),
				\end{align*}
		 where $\nu$ is the unit outward normal vector on the boundary $\partial B_r(x^0)$.

			 For any $0<\rho<\sigma<r_0$, we obtain that
			\begin{align*}
				&W(\sigma;u,x^0)-W(\rho;u,x^0)\\
				=&\int_{\rho}^{\sigma}\frac{\alpha(r)}{r^{n+2}(1-2\log r)^2}\int_{\partial B_r(x^0)}\left(\nabla u\cdot \nu -\frac{4\log r}{2\log r -1}\frac{u}{r}\right)^2d\mathcal{H}^{n-1}dr+\int_{\rho}^{\sigma}Q(r;u,x^0)dr.
			\end{align*}
	
	        Furthermore, for positive sequence $r_m$,

	        	\begin{align}\label{order-q}
	        		&W(r_m\sigma;u,x^0)-W(r_m\rho;u,x^0)\nonumber\\
	        		=&\int_{r_m\rho}^{r_m\sigma}\frac{\alpha(r)}{r^{n+2}(1-2\log r)^2}\int_{\partial B_r(x^0)}\left(\nabla u\cdot \nu -\frac{4\log r}{2\log r -1}\frac{u}{r}\right)^2d\mathcal{H}^{n-1 } dr
	        		+\int_{r_m\rho}^{r_m\sigma}Q(r;u,x^0)dr\nonumber\\
	        		=&\int_{B_{r_m\sigma}(x^0)\backslash B_{r_m\rho}(x^0)}\frac{\alpha(|x-x^0|)}{|x-x^0|^{n+2}(1-2\log |x-x^0|)^2}
	        		\left(\nabla u(x)\cdot \frac{x-x^0 }{|x-x^0|}-\frac{4\log |x-x^0|}{2\log |x-x^0| -1}\frac{u}{|x-x^0|}\right)^2dx\nonumber\\
	        		&+\int_{r_m\rho}^{r_m\sigma}Q(r;u,x^0)dr \nonumber\\
	        		=&\int_{B_{\sigma}(0)\backslash B_{\rho}(0)}r_m^n\frac{\alpha(|r_mz|)}{|r_mz|^{n+2}(1-2\log |r_m z|)^2}
	        		\left(\nabla u(r_mz+x^0)\cdot \frac{z}{|z|}-\frac{4\log |r_mz|}{2\log |r_mz| -1}\frac{u}{|r_mz|}\right)^2dz\nonumber\\
	        		&+\int_{r_m\rho}^{r_m\sigma}Q(r;u,x^0)dr \nonumber\\
	        		=&\int_{B_{\sigma}(0)\backslash B_{\rho}(0)}
	        		\frac{(1-2\log r_m)^2}{|z|^{n+4}}\frac{\alpha(|r_mz|)}{(1-2\log |r_m z|)^2}\left(\nabla u_{r_m}(z)\cdot z -\frac{2}{\alpha(|r_m z|)}u_{r_m}(z)\right)^2dz\nonumber\\
	        		&+\int_{r_m\rho}^{r_m\sigma}Q(r;u,x^0)dr.
	        	\end{align}
	
	       Furthermore, one has

	       	\begin{align*}
	       		&W(r_m\sigma;u,x^0)-W(r_m\rho;u,x^0)-\int_{r_m\rho}^{r_m\sigma}Q(r;u,x^0)dr\\
	       		=&\int_{B_{\sigma}(0)\backslash B_{\rho}(0)}
	       		\frac{1}{|z|^{n+4}}\frac{(1-2\log r_m)^2}{2\log |r_m z|\left(2\log |r_mz|-1\right)}\left(\nabla u_{r_m}(z)\cdot z -\frac{2}{\alpha(|r_m z|)}u_{r_m}(z)\right)^2dz\geq&0,
	       	\end{align*}
	       		for $ r_m<r_0\ll 1.$ And on the other hand, due to the facts that $\displaystyle \lim_{r\to 0+} W(r;u,x^0)$ exists and $Q(r;u,x^0)$ is integrable, we obtain that
	       $$W(r_m\sigma;u,x^0)-W(r_m\rho;u,x^0)-\int_{r_m\rho}^{r_m\sigma}Q(r;u,x^0)dr\to 0\qquad\text{as}\quad m\to \infty. $$
	Thanks to the following fact
	\begin{align*}
		\frac{(1-2\log r_m)^2}{2\log |r_m z|\left(2\log |r_mz|-1\right)}\to 1,\qquad\text{as} \quad m\to \infty, 
	\end{align*}
that we can arrive 
	\begin{align}\label{order}
		\displaystyle\lim_{m\to \infty} \int_{B_{\sigma}(0)\backslash B_{\rho}(0)}
		\frac{1}{|z|^{n+4}}\left(\nabla u_{r_m}(z)\cdot z -\frac{2}{\alpha(r_m z)}u_{r_m}(z)\right)^2dz= 0.
	\end{align}
	Observe that $\lim_{m\to \infty}\alpha(r_m z)=1$, we obtain that 
	\begin{align*}
		\lim_{m\to \infty }\frac{2}{\alpha(r_m z)}=2,
	\end{align*}
	and then $\nabla u_0(z)\cdot z -2 u_0(z) =0$ for any $z\in\mathbb{R}^n$.
Therefore, this yields the desired homogeneity of $u_0$ (2-homo) by the fact that $u_m$ converges to $u_0$ weakly in $W^{1,2}_{\text{loc}}(\mathbb{R}^n)$.

			Moreover, the homogeneity of ${{u}}_0$ gives that
		
				\begin{align}\label{eq2.4}
					\displaystyle\lim_{r\rightarrow 0+}W(r;u,x^0)
					=&\displaystyle\lim_{r\rightarrow 0+}\alpha(r)\int_{B_1}\frac{1}{2}|\nabla u_r|^2+{G}(r;u_r) dx -\int_{\partial B_1}u_r^2 d\mathcal{H}^{n-1}\nonumber\\
					=&\int_{B_1}\left(\frac{1}{2}|\nabla u_0|^2+u_0 \right)dx-\int_{\partial B_1}u_0^2d\mathcal{H}^{n-1}\nonumber\\
					=&\int_{B_1}\left(-\frac{1}{2}\Delta u_0 u_0+u_0\right) dx+\int_{\partial B_1}\frac{1}{2}\left(\nabla u_0\cdot \nu-2 u_0\right) u_0 d\mathcal{H}^{n-1}\nonumber\\
					=&\frac{1}{2}\int_{B_1}  u_0 dx\geq 0.
				\end{align}

			Additionally, it should be noted that of $W(0+;u,x^0)=0$, it is easy to obtain $u= 0$ in a small ball $B_{\delta}(x^0)$ for some $\delta>0$.
			
			(3) Although $W(r;u,x^0)$ loses monotonicity respect to $r$, $W(r;u,x^0)-\int_0^r Q(s;u,x^0)ds$ is monotonically increasing respect to $r$. Furthermore, since $Q(r;u,x^0)$ is integrable with respect to $r$, we can refer to \cite[Lemma 4.4 (v)]{vw12} to obtain that for each $\delta>0$,
			
			\begin{align*}
				W(0+;u,x^0)\leq W(r;u,x^0)-\int_0^r Q(s;u,x^0)ds \leq W(r;u,x^0)-\frac{\delta}{2}\leq W(0+;u,x^0)-\delta,
			\end{align*}	
			provided that we choose for fixed $x^0$ first $r>0$ and then $|x-x^0|$ small enough.
		The last inequality holds here due to our utilization of the absolute continuity of the Lebesgue integral. These imply the upper-semicontinuity of $W(0+;u,x^0)$ related to $x^0$.
		\end{proof}
		
		\begin{remark}\label{r2.8}
		Consider the governing equation for
		$u_r$ in $\{u_r>0\}$,
		\begin{align*}
			\Delta u_r=\frac{-\log u(x^0+rx)}{1-2\log r}=\frac{-\log \left(u_r\mu(r)\right)}{1-2\log r},\quad(\text{see} \eqref{u_r})
		\end{align*}
		where $\mu(r)=r^2(1-2\log r)$. If $u_r\to u_0$ weakly in $W^{1,2}(\Omega)$ as $r\to 0+$, taking limit $r\to 0+$ gives the property (2) homogeneity of the blow-up limits $u_0$, which implies that $u_0$ satisfies the following equation
			\begin{equation}\label{eq2.5}
			\Delta u_0=\chi_{\{u_0>0\}}\qquad\text{in}\quad \mathbb{R}^n.
			\end{equation}
			In fact, $\Delta u_0=0$ in $\{u_0=0\}$ is evident. This is nothing but the classical obstacle problem in whole domain, therefore, this finding greatly aids us in discussing the properties of blow-up limits as we do in Section 4.
		\end{remark}

		\vspace{10pt}
		\section{Epiperimetric inequality}
		This section is dedicated to show the \textit{epiperimetric inequality} which plays a crucial role in indicating an energy decay estimate and determining the uniqueness of the blow-up limit when analyzing the regularity of the free boundary. Firstly, let us observe the properties of the two integrable terms $Q(r;u,x^0)$ and $\overline{Q}(r;u,x^0)$. Recalling that $Q(r;u,x^0)$ as mentioned in Lemma \ref{monotonicity}, we check at once that Theorem \ref{lem1.1.}, which yields that $|Q(r;u,x^0)|\leq C\displaystyle\frac{-\log(-\log r)}{r(\log r)^2}$ for $0<r\ll 1$, moreover,
		$$	\overline{Q}(r;u,x^0)=Q(r;u,x^0)-\displaystyle\frac{1}{r(1-2\log r)^{1+\gamma}},\qquad\text{for any}\quad \gamma\in(0,1),$$ 
		these give that following properties.
		\begin{remark}\label{rem3.2}\ It's easy to check the following properties of integrable terms $Q(r;u,x^0)$ and $\overline{Q}(r;u,x^0)$.
			 
			\begin{itemize}
				\item $\displaystyle\lim_{r\to 0+}Q(r;u,x^0)=-\infty$;\quad
				 $\displaystyle\lim_{r\to 0+}\overline{Q}(r;u,x^0)=-\infty$;
				 \\[3pt]
				\item 	$\displaystyle\lim_{r\to 0+}\frac{Q(r;u,x^0)}{\frac{1}{r(1-2\log r)^{1+\gamma}}}=0$;
				\\[3pt]
				\item  $\displaystyle\lim_{r\to 0+}{\frac{Q(r;u,x^0)}{\overline{Q}(r;u,x^0)}}=0$.
			\end{itemize}
		\end{remark}
		\begin{remark}\label{rem3.3}
			The integrablity of $\overline{Q}(r;u,x^0) $ with respect to $r$ that implies that $\displaystyle\lim_{s\to 0}\int_{0}^{s}\overline{Q}(r;u,x^0)dr=0$.
		\end{remark}
		
		Secondly, noting that the boundary condition $v=\mathcal{C}$ on $\partial B_1$ shows the following energy relation for the minimizer $u$ (Remark \ref{r3.3}). This relation assists us in proving energy decay.
		\begin{remark}\label{r3.3}
			The facts that  $v=\mathcal{C}$ on $\partial B_1$ and $u$ is the minimizer of \eqref{eq1.0} gives that 
			$$M(s;u)\leq M(s;v)\quad \text{for any}\quad s\leq r_0,$$
			where $r_0$ such that $B_{r_0}(x^0)\subset \Omega$.
		\end{remark}
		
Thirdly, we present the \textit{Proof of Lemma \ref{epip}.}\ 

\textit{Proof of epiperimetric inequality \eqref{6.1}}.	We will provide the proof by contradiction. Assume that for any $ \eta\in(0,1)$ and $ \delta>0,$ there exists a non-negative function $\mathcal{C}\in W^{1,2}(B_1)$ that is homogeneous of degree 2 and some $h_{\nu}\in\mathbb{H}$ satisfying $$\|\mathcal{C}-h_{\nu}\|_{W^{1,2}(B_1)}\leq \delta,$$ such that there is $s\in (0, s_0(\delta)]$  satisfying
			\begin{equation}
				M(s;v)- M_0(h_{\nu})-\int_{0}^{s}\overline{Q}(r;u,x^0)dr > (1-\eta) \left[M(s;\mathcal{C})- M_0(h_{\nu})-\int_{0}^{s}\overline{Q}(r;u,x^0)dr\right],\nonumber
			\end{equation}
          for any $v \in \mathcal{C}+W^{1,2}_0(B_1).$
			Without loss of generality, one can assume that there exist sequences $\eta_k\rightarrow 0$, $\delta_k\rightarrow 0$, $\mathcal{C}_k\in W^{1,2}(B_1)$ and $h_{\nu_{k}}\in \mathbb{H}$ such that non-negative $\mathcal{C}_k\in W^{1,2}(B_1)$ is a homogeneous global function of degree 2 satisfying
			\begin{align}\label{h1}
				\|\mathcal{C}_k-h_{\nu_{k}}\|_{W^{1,2}(B_1)}=\inf_{h_{\nu}\in\mathbb{H}}\|\mathcal{C}_k-h_{\nu}\|_{W^{1,2}(B_1)}=\delta_k,
			\end{align}
			and $\exists \ s_k \in (0, s_0(\delta_k))$  (where $s_0(\delta)$ also satisfies $-\log s_0(\delta)<\delta^{-\frac{2}{\gamma}}$) such that
				\begin{equation}\label{eq3.2}
				M(s_k;v)- M_0(h_{\nu_{k}})-\int_{0}^{s_k}\overline{Q}(r;u,x^0)dr > (1-\eta_k) \left[M\left(s_k;\mathcal{C}_k\right)- M_0(h_{\nu_{k}})-\int_{0}^{s_k}\overline{Q}(r;u,x^0)dr\right],
			\end{equation}
			for $v\in \mathcal{C}_k+W^{1,2}_0(B_1).$
			For simplicity, rotating in $\mathbb{R}^n$ if necessary, we may assume that
			\begin{align}\label{h2}
				h_{\nu_{k}} =h_{e_n}=\frac{1}{2}\max (x_n,0)^2=:h,
			\end{align}
			it shows that $\displaystyle\chi_{\{h>0\}}=\chi_{\{x_n>0\}}$.

		According to the inequality \eqref{eq3.2}, it is easy to check that
			\begin{equation}\label{eq3.3}
				(1-\eta_k) \left[M\left(s_k;\mathcal{C}_k\right)- M_0(h)-\int_{0}^{s_k}\overline{Q}(r;u,x^0)dr\right]<M(s_k;v)- M_0(h)-\int_{0}^{s_k}\overline{Q}(r;u,x^0)dr,
			\end{equation}
			for every $v\in W^{1,2}(B_1)$ and $v=\mathcal{C}_k$ on $\partial B_1$.
			
			Since $\displaystyle\Delta h=\chi_{\{h>0\}},$ one gets
			\begin{equation}
				\int_{B_1}-\nabla h\cdot\nabla {\phi} \ dx+\int_{\partial B_1}\phi \nabla h\cdot \nu \ d\mathcal{H}^{n-1}-\int_{B_1}  {\phi}\chi_{\{h>0\}} dx=0, \quad \text{for any }\quad  {\phi}\in W^{1,2}(B_1).\nonumber
			\end{equation} 
			Then taking $\phi(x)=z(x)-h(x)$ in the above equation, and for simplicity, we denote
			\begin{equation*}
				\mathscr{A}(z):= \int_{B_1}\nabla h \cdot \nabla(z-h)\ + (z-h) \chi_{\{x_n>0\}} dx-\int_{\partial B_1}2M_0(z-h)\ d\mathcal{H}^{n-1},
			\end{equation*}
			and $\mathscr{A}(z)=0$ for any $z\in W^{1,2}(B_1)$. Then, in conjunction with the inequality \eqref{eq3.3}, this provides,
			\begin{align*}
				&(1-\eta_k)\left[M\left(s_k;\mathcal{C}_k\right)-M_0(h)-\int_{0}^{s_k}\overline{Q}(r;u,x^0)dr-\mathscr{A}(\mathcal{C}_k)\right]\\
				<&M(s_k;v)-M_0(h)-\int_{0}^{s_k}\overline{Q}(r;u,x^0)dr-\mathscr{A}(v).\nonumber
			\end{align*}
			Specifically,
				\begin{align*}
					&(1-\eta_k)\Bigg[\alpha(s_k)\int_{B_1}\frac{1}{2}|\nabla \mathcal{C}_k|^2+{G}\left(s_k;\mathcal{C}_k\right)\ dx-\int_{\partial B_1}\mathcal{C}_k^2\  d\mathcal{H}^{n-1} \\
					&-\int_{B_1}\frac{1}{2}|\nabla  h|^2+ h \ dx+\int_{\partial B_1}h^2 d\mathcal{H}^{n-1}-\int_{0}^{s_k}\overline{Q}(r;u,x^0)dr\\
					&-\int_{B_1} \nabla h\cdot \nabla(\mathcal{C}_k-h)+(\mathcal{C}_k-h)\chi_{\{x_n>0\}}\  dx
					+2\int_{\partial B_1}h\cdot(\mathcal{C}_k-h) \ d\mathcal{H}^{n-1} \Bigg]\\
					<&\alpha(s_k)\int_{B_1}\frac{1}{2}|\nabla v|^2+{G}(s_k;v)dx -\int_{\partial B_1}v^2\  d\mathcal{H}^{n-1}- \int_{B_1}\frac{1}{2}|\nabla h|^2+h\ dx +\int_{\partial B_1}h^2 d\mathcal{H}^{n-1}\\
					&-\int_{0}^{s_k}\overline{Q}(r;u,x^0)dr-\int_{B_1}\nabla h\cdot\nabla(v-h)+(v-h)\chi_{\{x_n>0\}}\ dx
					+2\int_{\partial B_1}h(v-h)d\mathcal{H}^{n-1}.
				\end{align*}
			This yields that
			\begin{equation}\label{eq3.4}
				\begin{aligned}
					&(1-\eta_k)\Bigg[ \frac{1}{2}\int_{B_1}|\nabla(\mathcal{C}_k-h)|^2 \ dx-	\int_{\partial B_1}(\mathcal{C}_k-h)^2\  d\mathcal{H}^{n-1}+\int_{B_1}{G}\left(s_k;\mathcal{C}_k\right)-\mathcal{C}_k\chi_{\{h>0\}}\ dx\\
					&-\frac{1}{2\log {s_k}}\int_{B_1}\frac{1}{2}|\nabla \mathcal{C}_k|^2+{G}\left(s_k;\mathcal{C}_k\right)\ dx -\int_{0}^{s_k}\overline{Q}(r;u,x^0)dr \Bigg] \\
					<&\int_{B_1}\frac{1}{2}|\nabla(v-h)|^2 \ dx-\int_{\partial B_1}|v-h|^2\  d\mathcal{H}^{n-1}
					+\int_{B_1}{G}(s_k;v)-v\chi_{\{h>0\}} dx\\
					&-\frac{1}{2\log {s_k}}\int_{B_1}\frac{1}{2}|\nabla v|^2+{G}(s_k;v)\ dx -\int_{0}^{s_k}\overline{Q}(r;u,x^0)dr .
				\end{aligned}
			\end{equation}

			Define $w_k:=(\mathcal{C}_k - h)/{\delta_k}$, such that $\|w_k\|_{W^{1,2}(B_1)}=1$. To establish the desired result, it is now necessary to demonstrate that $w_k$ converges strongly to $w$ in $W^{1,2}(B_1)$, with the additional condition that  $w= 0$ in $B_1$, ultimately resulting in a contradiction.
			
			We divide both sides of the inequality \eqref{eq3.4} by $\delta_k^2$, it follows from that $w_k=(\mathcal{C}_k - h)/{\delta_k}$, which shows the following inequality regarding $w_k$,
				\begin{align}\label{eq3.5}
					&(1-\eta_k)\Bigg[ \frac{1}{2}\int_{B_1}|\nabla w_k|^2 \ dx-	\int_{\partial B_1} w_k^2\  d\mathcal{H}^{n-1}+\frac{1}{\delta_k^2}\int_{B_1}{G}\left(s_k;\mathcal{C}_k\right)-\mathcal{C}_k\chi_{\{h>0\}}\ dx\nonumber\\
					&-\frac{1}{2\delta_k^2\log {s_k}}\int_{B_1}\frac{1}{2}|\nabla \mathcal{C}_k|^2+{G}\left(s_k;\mathcal{C}_k\right)\ dx -\frac{1}{\delta_k^2}\int_{0}^{s_k}\overline{Q}(r;u,x^0)dr \Bigg] \nonumber\\
					<&\int_{B_1}\frac{1}{2}\left|\nabla(\frac{v-h}{\delta_k})\right|^2 \ dx-\int_{\partial B_1}\left(\frac{v-h}{\delta_k}\right)^2\  d\mathcal{H}^{n-1}
					+\frac{1}{\delta_k^2}\int_{B_1}{G}(s_k;v)-v\chi_{\{h>0\}} dx\nonumber\\
					&-\frac{1}{2\delta_k^2\log {s_k}}\int_{B_1}\frac{1}{2}|\nabla v|^2+{G}(s_k;v)\ dx -\frac{1}{\delta_k^2}\int_{0}^{s_k}\overline{Q}(r;u,x^0)dr.
				\end{align}
				
			Due to $\|{w}_k\|_{W^{1,2}(B_1;\mathbb{R}^m)}=1$, namely, ${w}_k$ is bounded in $W^{1,2}(B_1)$, this yields that there is a weakly convergent subsequence, still denoted by ${w}_k$, such that ${w}_k\rightharpoonup {w}$ in $W^{1,2}(B_1)$ for some ${w}$ belonging to $W^{1,2}(B_1)$, then we aim to demonstrate ${w}_k\to {w}$ in $W^{1,2}(B_1)$, with the additional condition that ${w}= 0$ in $B_1$. Our proof is constructed due to the following four claims.

			Claim 1: ${w}= 0$ in $B_1^-:=\{x\in B_1: x_n<0\}$.
			
			To see this, let $v:=(1-\xi)\mathcal{C}_k+\xi h$ in \eqref{eq3.5}, where $\xi(x)=\xi(|x|) \in C_0^\infty(B_1)$, $0\leq\xi \leq 1$ in $B_1$ and $\displaystyle \int_{B_1} \xi dx=1$. Since $\displaystyle\frac{v-h}{\delta_k}=(1-\xi)(\displaystyle\frac{\mathcal{C}_k-h}{\delta_k})=(1-\xi)w_k,$ then the inequality \eqref{eq3.5} implies
				\begin{align}\label{eq3.6}
					&(1-\eta_k)\Bigg[ \int_{B_1}\frac{1}{2}|\nabla(w_k)|^2 \ dx- \int_{\partial B_1} w_k^2\  d\mathcal{H}^{n-1} +\int_{B_1}\frac{1}{\delta_k^2}G\left(s_k;\mathcal{C}_k\right)- \mathcal{C}_k \chi_{\{h>0\}} \ dx\nonumber\\
					&
					-\frac{1}{2\delta_k^2\log s_k}\int_{B_1}\frac{1}{2}|\nabla(\mathcal{C}_k)|^2+G\left(s_k;\mathcal{C}_k\right) dx-\frac{1}{\delta_k^2}\int_{0}^{s_k}\overline{Q}(r;u,x^0)dr \Bigg] \nonumber\\
					<&\int_{B_1}\frac{1}{2}|\nabla((1-\xi)w_k)|^2 \ dx-\int_{\partial B_1}\left((1-\xi)w_k\right)^2\  d\mathcal{H}^{n-1}
					+\frac{1}{\delta_k^2}\int_{B_1}  {G}(s_k;v)- v \chi_{\{h>0\}}        \ dx\nonumber\\
					&
					-\frac{1}{2\delta_k^2\log s_k}\int_{B_1}\frac{1}{2}|\nabla v|^2+{G}(s_k;v) dx-\frac{1}{\delta_k^2}\int_{0}^{s_k}\overline{Q}(r;u,x^0)dr.
				\end{align}

			Notice that
			 \begin{align*}
				&\int_{B_1}|\nabla((1-\xi)w_k)|^2 dx-2\int_{\partial B_1}|(1-\xi)w_k|^2 \ d\mathcal{H}^{n-1}\\
				&-(1-\eta_k)\left(\int_{B_1}|\nabla w_k|^2\ dx-2\int_{\partial B_1}|w_k|^2 d\mathcal{H}^{n-1} \right)\\
				\leq &\int_{B_1}|\nabla((1-\xi)w_k)|^2 dx+2(1-\eta_k)\int_{\partial B_1}|w_k|^2d\mathcal{H}^{n-1}\\
				\leq&C_1,
			\end{align*}
			where $C_1$ is a positive uniform constant, on account of $\|w_k\|_{ {W^{1,2}}(B_1)}=1$.  Moreover, recalling that $\mathcal{C}_k\in W^{1,2}(B_1)$, $v\in W^{1,2}(B_1)$ and $\left(-\log s_0(\delta_k)\right)^{-1}=o(\delta_k^{\frac{2}{\gamma}})$ implies that $\displaystyle\frac{1}{2\delta_k^2\log s_k} \to 0$ as $k\to \infty$, these imply that 
			\begin{align*}
				(1-\eta_k)\frac{1}{2\delta_k^2\log s_k}\int_{B_1}\frac{1}{2}|\nabla \mathcal{C}_k|^2+{G}\left(s_k;\mathcal{C}_k\right)\ dx-\frac{1}{2\delta_k^2\log s_k}\int_{B_1}\frac{1}{2}|\nabla v|^2+{G}(s_k;v)\ dx\leq C_2,
			\end{align*}
			for sufficiently large $k$. 
			Again, based on the choice of $s_0(\delta_k)$ such that $\left(-\log s_0(\delta_k)\right)^{-1}=o(\delta_k^\frac{2}{\gamma})$ as $k\to\infty$, we also have regarding the integrable term $\overline{Q}$,
		 \begin{align*}
					\displaystyle\lim_{k\to \infty} \frac{1}{\delta_k^2}\int_{0}^{s_k}\overline{Q}(r;u,x^0)dr =0;
				\end{align*}
			 \begin{align}\label{Qde}
					\eta_k\displaystyle\frac{1}{\delta_k^2}\displaystyle\int_{0}^{s_k}\overline{Q}(r;u,x^0)dr \to 0 \quad\text{as}\quad k\to \infty.
				\end{align}
				Hence, for sufficiently large $k$,$$  \eta_k\displaystyle\frac{1}{\delta_k^2} \displaystyle\int_{0}^{s_k}\overline{Q}(r;u,x^0)dr\leq C_3,$$
				 where $C_3$ is a positive constant.

			Moreover, recalling that $\left(-\log s_0(\delta_k)\right)^{-1}=o(\delta_k^{\frac{2}{\gamma}})$ and $\gamma\in(0,1)$, it shows that $\left(-\log s_k\right)^{-1}=o(\delta_k^2)$. And the facts that $\mathcal{C}_k\in W^{1,2}(B_1)$ and $v\in W^{1,2}(B_1)$, it follows that for sufficiently large $k$,
			\begin{align*}
					-\frac{1}{2\log s_k}\int_{B_1} \frac{1}{2}|\nabla \mathcal{C}_k|^2+{G}\left(s_k;\mathcal{C}_k\right) dx &=o(\delta_k^2),\\
					\text{and}\qquad\qquad\qquad\qquad\qquad\qquad\qquad\qquad\qquad\\
					-\frac{1}{2 \log s_k}\int_{B_1} \frac{1}{2}|\nabla  v|^2+{G}(s_k;v) dx&=o(\delta_k^2).
			\end{align*}
			Therefore, let's focus on the remaining part,
				\begin{align*}
					(1-\eta_k)\Bigg[\int_{B_1} G\left(s_k;\mathcal{C}_k\right)- \mathcal{C}_k \chi_{\{h>0\}}   \ dx
				\Bigg] <C_4 \delta_k^2+\int_{B_1} {G}(s_k;v)-v\chi_{\{h>0\}}\ dx,
				\end{align*}
			where $C_4=\max \left\{C_1, C_2,C_3\right\}$ is a positive constant.
		
			Let us now observe that $F(u)=u( -\log u+1)$ not only lacks the scaling property (see \eqref{noscaling}) as discussed in \cite{w99} and \cite{w00}, but also lacks convexity as mentioned in \cite{w00,dz}. To address the challenges posed by these properties in calculations, we have to develop a method to analyze ${G}(s_k;v)$ as $k\to \infty$,
			
			\begin{equation}\label{eq3.7}
				{G}(s_k;v)=v+o\left((1-2\log s_k)^{-\gamma}\right),
			\end{equation}
			where $``o"$ represents higher-order infinitesimal term, i.e.,
			\begin{equation*}
				\displaystyle\lim_{k\to \infty}\frac{{G}(s_k;v)-v}{(1-2\log s_k)^{-\gamma}}=0.
			\end{equation*}
			
			This approach yields the result that
		  	\begin{align*}
		  		&(1-\eta_k)\left[\int_{ B_1} \mathcal{C}_k + o\left((1-2\log s_k)^{-\gamma}\right)\ dx-\int_{ B_1^+}\mathcal{C}_k \ dx\right]\\
		  		<\ 
		  		&C_4\delta_k^2+\int_{ B_1} v + o\left((1-2\log s_k)^{-\gamma}\right)dx-\int_{ B_1^+}vdx,
		  	\end{align*}
			then, the above inequality can be simplified to
			 
			 		\begin{align*}
					\int_{ B_1^-}(\xi-\eta_k)\mathcal{C}_k \ dx< C_4\delta_k^2+\eta_k o\left((1-2\log s_k)^{-\gamma}\right).
				\end{align*}

			Next, similarly, by dividing both sides of the inequality by $\delta_k^2$, we can obtain the following inequality,
			\begin{equation*}
				\int_{ B_1^-}(\xi - \eta_k)\frac{w_k}{\delta_k} dx<C_4+\eta_k o(1)\leq C,
			\end{equation*}
			for sufficiently large $k$, where $C\geq C_4$. Furthermore, observe that the test function $\xi$ is a radial function, we can obtain that
			
			\begin{equation*}
				\int_{0}^1(\xi(\rho) - \eta_k)\int_{\partial B_{\rho}^-}\frac{w_k}{\delta_k} d\mathcal{H}^{n-1}d \rho\leq C.
			\end{equation*}
			As $k\to\infty$, we have that $\eta_k\to 0$ and $\delta_k\to 0$. It implies that $w=0$ in $B_1^-$.

			Claim 2: $\Delta w=0$ in $B_1^+$.
			
			To see this, one can select a ball $B_0\Subset B_1^+(0)$, and reassign $ {v}:=(1-\xi) {c}_k+ \xi( {h}+\delta_k {g})$ in \eqref{eq3.4}, where $\xi\in C_0^{\infty}(B_1^+)$ and $ {g}\in W^{1,2}(B_1)$ such that $\xi=1$ in $B_0$, $\xi=0$ in $B_1^-$, $0\leq\xi\leq1$ in $B_1^+\setminus B_0$. On the one hand,  It follows that $v|_{\partial B_1}=\mathcal{C}|_{\partial B_1}$ and $v=\mathcal{C}$ in $B_1^-$. on the other hand, due to the expression $h={G}(s_k;h)+o((1-2\log  s_k)^{-\gamma})$, as $k\to \infty$ and $s_k\leq s_0(\delta_k)$. therefore, the inequality \eqref{eq3.5} arrives at

	 	\begin{align*}
	 	& \frac{1}{2}\int_{B_1}|\nabla w_k|^2 \ dx-\frac{1}{2\delta_k^2\log {s_k}}\int_{B_1}\frac{1}{2}|\nabla \mathcal{C}_k|^2+{G}\left(s_k;\mathcal{C}_k\right)\ dx -\frac{1}{\delta_k^2}\int_{0}^{s_k}\overline{Q}(r;u,x^0)dr +\frac{1}{\delta_k^2}o((1-2\log  s_k)^{-\gamma})\nonumber\\
	 	<&\frac{1}{2} \int_{B_1}|\nabla((1-\xi)w_k+\xi {g})|^2 \ dx
	 	-\frac{1}{2\delta_k^2\log {s_k}}\int_{B_1}\frac{1}{2}|\nabla v|^2+{G}(s_k;v)\ dx-\frac{1}{\delta_k^2}\int_{0}^{s_k}\overline{Q}(r;u,x^0)dr\\ & \frac{1}{\delta_k^2}o((1-2\log  s_k)^{-\gamma})+o(1).
	 \end{align*}
	Moreover, recalling that \eqref{Qde}, $\mathcal{C}_k$, $v\in W^{1,2}(B_1)$, and $\left(-\log s_k(\delta_k)\right)^{-1}=o(\delta_k^{\frac{2}{\gamma}})$, it implies that 
	 
				\begin{align*}
						\int_{B_1^+}\frac{1}{2}|\nabla {w}_k|^2\ dx 
					< \int_{B_1^+}|\nabla((1-\xi)w_k+\xi {g})|^2 \ dx+
					o(1).
				\end{align*}
			As $k\to \infty$, indicating 
			\begin{equation}\label{eq3.8}
				\int_{B_1}|\nabla w|^2 dx\leq\int_{B_1}|\nabla((1-\xi )w +\xi g)|^2 dx	.	
			\end{equation}
 Recalling the choice of the auxiliary function $\xi$, we can obtain
			$$\int_{B_0}|\nabla w|^2 dx\leq \int_{B_0}|\nabla g|^2 dx,$$
		for all $g\in W^{1,2}(B_1)$ coinciding with $w$ on $\partial B_0$, and based on the selection of $B_0\Subset B_1^+(0)$, it follows that $w$ is the minimizer of the Dirichlet energy $\int_{ B_1^+}|\nabla g|^2 dx$ in the class $$\{g\in W^{1,2}(B_1); \quad g=w \quad\text{on }\quad \partial B_1^+\},$$ which corresponds to the Euler-Lagrange equation $\Delta w = 0$ in $B_1^+$.

			Claim 3: ${w}=0$ in $B_1$.
			
	For the verification of Claim 3, the procedure closely follows Step 3 of \cite[Theorem 3.1]{w00} and \cite[Theorem 1]{asu}. Here, we present a concise overview of the proof process.
	
	Given that $w= 0$ in $B_1^-$ (Claim 1) and $\Delta w=0$ in $B_1^+$ (Claim 2), and $w$ is a homogeneous harmonic function of degree 2. Initially, by odd extension, we extend $w$ to a homogeneous function of degree 2. According to Liouville's theorem, $w=\sum_{j=1}^{n-1}a_{nj}x_jx_n$ in $B_1^+$. Recalling that we have chosen $h$ as the minimizer of $\inf_{h_{\nu}\in\mathbb{H}}\|\mathcal{C}_k-h_{\nu}\|_{W^{1,2}(B_1)}$ (see \eqref{h1} and \eqref{h2}), leading to the deduction that $w= 0$ in $B_1^+$.

			Claim 4: $ {w}_k\to {w}$ strongly in $W^{1,2}(B_1)$.
			
			To illustrate this, considering the strong convergence $w_k\to w$ in $L^2(B_1)$, with $w= 0$ in $B_1$, our task is to establish the estimate $\|\nabla w_k\|_{L^2(B_1)}$. Let's define ${v}=(1-\xi) \mathcal{C}_k+\xi {h}$ with
			\begin{align*}
				\begin{split}
					\zeta(x)=\xi(|x|)=\left\{
					\begin{array}{lr}
						0,                 &|x|\geq  1,\\
						2-2|x|,                 &1/2<|x|\leq 1,\\
						1,         &|x|\leq \frac{1}{2},\\
					\end{array}
					\right.
				\end{split}
			\end{align*}
		that is $\xi(|x|)=\min\left(2\max(1-|x|,0),1\right)$.
			
		Noting $\displaystyle\frac{{v}-{h}}{\delta_k}=(1-\xi)w_k$, akin to the derivation of inequality \eqref{eq3.8}, one also derives the following
				\begin{align*}
					& \int_{B_1}|\nabla w_k|^2 dx \\
					<&\int_{B_1}|\nabla((1-\xi)w_k)|^2 dx +o(1)\\
					=&\int_{B_1} (1-\xi)^2|\nabla w_k|^2-2(1-\xi)(\nabla w_k\cdot \nabla\xi)w_k+|\nabla \xi|^2 w_k^2 dx +o(1).
				\end{align*}
        We can further simplify the above inequality to obtain
        	\begin{equation*}
        		 \left(1-(1-\xi)^2\right)\int_{B_1}|\nabla w_k|^2 dx < \int_{B_1} -2(1-\xi)(\nabla w_k\cdot \nabla\xi)w_k+|\nabla \xi|^2 w_k^2 dx + o(1).
        \end{equation*}
        Next, on one hand, the definition of $\xi$ shows that
        
        \begin{equation*}
        	\left(1-(1-\xi)^2\right)\int_{B_1}|\nabla w_k|^2 dx \geq \int_{B_{1/2}}|\nabla w_k|^2 dx.
        \end{equation*}
        On the other hand, since $w_k$ weakly converges to $w$ in $W^{1,2}(B_1)$, the right-hand side of the inequality tends to $0$.
        
        Finally, leveraging the homogeneity of $w_k$, we can conclude
        
				\begin{equation*}
					\int_{B_1}|\nabla w_k|^2 dx =2^{n+2}\int_{B_{1/2}}|\nabla w_k|^2 dx\to 0,\quad \text{as}\quad k\to \infty.
			\end{equation*}
		As a consequence, $w_k\to 0$ strongly in $W^{1,2}(B_1)$, which contradicts the fact that $\|w_k\|_{W^{1,2}(B_1)}=1$.
			\qed

		\vspace{25pt}

		\section{Energy decay and uniqueness of blow-up limits }
		In this section, we will estimate the energy decay and convergence rate of scaled solution to its blow-up limit. Before proceeding, it is indeed necessary to establish some preliminary results. Firstly, let us explore the intriguing outcomes related to homogeneous solutions in the classical context.

		\begin{definition}\label{global}
			$v\in W^{1,2}(\mathbb{R}^n)$ is called \textit{a homogeneous global solution of degree 2} (can be simply noted as the 2-homo solution), provided that it satisfies the following two conditions.
			
			(1) ({\it Homogeneity})
			$$ v(\lambda x)=\lambda^2v(x) \ \ \  \text{for all } \ \ \lambda>0 \ \ \text{and} \ \  x\in \mathbb{R}^n.$$
			
			(2) It is a weak solution to
			\begin{equation}\label{eq4.1}
				\Delta v= \chi_{\{v>0\}}\quad \text{in}\quad \mathbb{R}^n.
			\end{equation}
		\end{definition} 
		
			Next, we will show the characterization of $\mathcal{R}_u$ (Proposition \ref{regular} ). The proof will follow from \cite[Lemma 2, Proposition 3]{w99}, whose statements we recall below. The theorem demonstrates that the significant properties of 2-homo solutions.

		\renewcommand{\themythm}{\Alph{mythm}}
		\begin{mythm}\label{isolated} The half space solutions are (in the $W^{1,2}(B_1)$-topology) isolated with the class of homogeneous global solutions of degree 2, which means that any 2-homo solutions (say $u_m$) such that $u_m\to u_0$ in $W^{1,2}(B_1)$ and 
			\begin{equation*}
				0<\inf_{h_{\nu}\in\mathbb{H}} \|u_m -h_{\nu}\|_{W^{1,2}(B_1)}\to 0,\quad \text{as}\quad m\to 0,
			\end{equation*}
			then, $\inf_{h_{\nu}\in\mathbb{H}} \|u_0 -h_{\nu}\|_{W^{1,2}(B_1)}=0$. Additionally, let u be a homogeneous global solution of degree 2 as in Definition \ref{global}, the energy density
			\begin{equation}\label{eq4.2}
				M_0(u)\geq  \frac{\omega_n}{2}.
			\end{equation}
			Moreover, $\displaystyle M_0(u)=\frac{\omega_n}{2}$ implies that $u\in \mathbb{H}.$
		\end{mythm}

		As \cite[Corollary 2]{w99}, we use the similar energy density characterization of the regular part of the free boundary.
		\begin{proposition}{\it(characterization of $\mathcal{R}_u$)}\label{regular}
			A point $x^0$ is a regular free boundary point for $u$ provided that
			$$x^0\in\mathscr{F}(u)=\partial\{ u>0\} \cap \Omega\qquad \text{and} \qquad \displaystyle\lim_{r\rightarrow 0}W(r;u,x^0)=\displaystyle\frac{\omega_n}{2}.$$
			
		\end{proposition}
	    In fact, the regular points defined here by energy density are equivalent to the forms of the blow-up limits of solutions given earlier in page 9 for $\mathcal{R}_u$. This equivalence can be derived from \eqref{alp_n} and the Theorem \ref{isolated}.
		
		Therefore, we provide a way to characterize regular points due to the energy density. With this characterization, we can utilize energy decay to establish the uniqueness of the blow-up limit. In other words, there is an equivalence between the blow-up limit and the energy density limit.
		
		\begin{remark}
			In this work, to investigate the regularity of free boundary, we mainly consider the regular set $\mathcal{R}_{u}$. As for the singular set of free boundary points, the analysis of regularity is still open up to now and this would be our future research direction.
		\end{remark}
		\begin{remark}
			Theorem \ref{isolated}, along with the definition of $\mathcal{R}_u$, implies that the set of regular free boundary points $\mathcal{R}_u$ is open with respect to $\mathscr{F}(u)$.
		\end{remark}

	Subsequently, we demonstrate an energy decay estimate via the epiperimetric inequality. Building upon this energy decay analysis, we establish the uniqueness of the blow-up limit.	Here, the energy has subtracted the antiderivative of an integrable function, which is key to our handling of the term involving the logarithmic function.

	\textit{Proof of Proposition \ref{uniqueness}.}
	Initially, we embark on an analysis of the decay rate of  $\overline{W}(r;u,x^0)$  by establishing a linear inequality for $\displaystyle\frac{d}{dr}\overline{W}(r;u,x^0)$. Therefore, we first let  $\mathscr{E}(r):=\overline{W}(r;u,x^0)-\overline{W}(0+;u,x^0)$, then 
	
			\begin{align}\label{e(r)}
					\displaystyle \mathscr{E}(r)=&\frac{\alpha(r)}{r^{n+2}(1-2\log r)^{2}}\int_{B_r(x^0)} \frac{1}{2}|\nabla u|^2+F(u)dx-\frac{1}{r^{n+3}(1-2\log r)^{2}}\int_{\partial B_r(x^0)}u^2 d\mathcal{H}^{n-1}\nonumber\\
				&-\int_{0}^{r}\overline{Q}(s;u,x^0)ds-W(0^+;u,x^0)\nonumber\\
				=&\alpha(r)I(r)-J(r)-\int_{0}^{r}\overline{Q}(s;u,x^0)ds-W(0^+;u,x^0),	
			\end{align}
		where for convenience, we define
		 $$ I(r):=\frac{1}{r^{n+2}(1-2\log r)^{2}}\int_{B_r(x^0)} \frac{1}{2}|\nabla u|^2+F(u)dx,$$
		 and$$ J(r):=\frac{1}{r^{n+3}(1-2\log r)^{2}}\int_{\partial B_r(x^0)}u^2 d\mathcal{H}^{n-1},$$
		 where	
		 \begin{equation*}
		 	F(u)=	u( -\log u+1),     \qquad u\geq 0.
		\end{equation*}
		  Next, a series of direct computation gives that
		 	 
			\begin{align}\label{eq4.5}
				\mathscr{E}'(r)=&\left(\alpha(r)I(r)\right)'-J'(r)-\overline{Q}(r;u,x^0)\nonumber\\
				=&-\frac{1}{r}(n+2)\alpha(r)I(r)-\frac{1}{r}\left(\frac{1-4\log r}{\log r(1-2\log r)}\right)\alpha(r)I(r)+\frac{\alpha(r)}{r}\int_{\partial B_1}\frac{1}{2}|\nabla u_r|^2 d\mathcal{H}^{n-1}\nonumber\\
				&+\frac{\alpha(r)}{r}\int_{\partial B_1}{G}(r;u_r)d\mathcal{H}^{n-1}-J'(r)-\overline{Q}(r;u,x^0),
					\end{align}
					where  \begin{equation*}
						{G}(r;u_r)=\displaystyle\frac{u_r}{1-2\log r}\left[-\log \left(u_rr^2(1-2\log r)\right) +1\right],\qquad u_r\geq 0.
					\end{equation*}
				Here, we firstly note that
				$$\int_{\partial B_1}|\nabla u_r|^2d\mathcal{H}^{n-1}=\int_{\partial B_1} |\nabla_{\mathbf{n}} u_r|^2+|\nabla_{\theta} u_r|^2d\mathcal{H}^{n-1},$$
				and
				 $$\int_{\partial B_1}|\nabla_{\theta}u_r|^2 d\mathcal{H}^{n-1}=\int_{\partial B_1}|\nabla_{\theta}\mathcal{C}_r|^2 d\mathcal{H}^{n-1},$$
				 where $\nabla_{\theta}u_r:=\nabla u_r- (\nabla u_r\cdot \mathbf{n})\mathbf{n}$ as the surface derivative of $u_r$ and $\mathbf{n}$ denotes the topological outward normal of $\partial B_1$.

				  These imply that
				 
					\begin{align}	\label{eq4.6}
						\int_{\partial B_1}\frac{1}{2}|\nabla u_r|^2+{G}(r;u_r)d\mathcal{H}^{n-1}=&\int_{\partial B_1}\frac{1}{2}|\nabla_{\mathbf{n}} u_r|^2+\frac{1}{2}|\nabla_{\theta} u_r|^2+{G}(r;u_r)d\mathcal{H}^{n-1}\nonumber\\
						=&\int_{\partial B_1}\frac{1}{2}|\nabla_{\mathbf{n}} u_r|^2+\frac{1}{2}|\nabla_{\theta} \mathcal{C}_r|^2+{G}(r;u_r)d\mathcal{H}^{n-1}\nonumber\\
						=&\int_{\partial B_1}\frac{1}{2}|\nabla_{\mathbf{n}} u_r|^2+\frac{1}{2}|\nabla \mathcal{C}_r|^2-\frac{1}{2}|\nabla_{\mathbf{n}} \mathcal{C}_r|^2+{G}(r;u_r)d\mathcal{H}^{n-1} \nonumber\\
						=&\int_{\partial B_1}\frac{1}{2}|\nabla_{\mathbf{n}} u_r|^2+\frac{1}{2}|\nabla \mathcal{C}_r|^2-2 \mathcal{C}_r^2+{G}(r;u_r)d\mathcal{H}^{n-1}\nonumber\\
						=&\int_{\partial B_1}\frac{1}{2}|\nabla \mathcal{C}_r|^2+{G}(r;u_r)d\mathcal{H}^{n-1}+ \int_{\partial B_1}-2 \mathcal{C}_r^2+\frac{1}{2}|\nabla_{\mathbf{n}} u_r|^2d\mathcal{H}^{n-1} \nonumber\\
						=& (n+2)\int_{ B_1}\frac{1}{2}|\nabla \mathcal{C}_r|^2+{G}(r;\mathcal{C}_r)dx-\frac{2}{n+2}\int_{\partial B_1}\frac{\mathcal{C}_r}{1-2\log r}d\mathcal{H}^{n-1}\nonumber\\
						& + \int_{\partial B_1}-2 \mathcal{C}_r^2+\frac{1}{2}|\nabla_{\mathbf{n}} u_r|^2d\mathcal{H}^{n-1} ,
						\end{align}
						where $\nabla_{\mathbf{n}} u_r:=( \nabla u_r\cdot \mathbf{n})\mathbf{n}$ and $\mathbf{n}$ denotes the topological outward normal of $\partial B_1$.
					Moreover,
					\begin{align}\label{eq4.7}
						J'(r)=\left(\int_{\partial B_1} u_r^2 d\mathcal{H}^{n-1}\right)'=\frac{2}{r}\int_{\partial B_1}u_r\left(\nabla u_r\cdot x-\frac{2}{\alpha(r)}u_r\right)d\mathcal{H}^{n-1}.
					\end{align}
				Then, substituting \eqref{eq4.6} - \eqref{eq4.7} into \eqref{eq4.5}, we get that
			
				\begin{align*}	
			\mathscr{E}'(r)=&\frac{n+2}{r}\left(M(r;\mathcal{C}_r)-M(r;u_r)\right)+\frac{\alpha(r)}{r}\int_{\partial B_1}-2 \mathcal{C}_r^2+\frac{1}{2}|\nabla_{\mathbf{n}} u_r|^2d\mathcal{H}^{n-1}\\
			&-\int_{\partial B_1}\frac{2}{r}u_r\nabla u_r\cdot xd\mathcal{H}^{n-1}+\int_{\partial B_1}\frac{4}{r\alpha(r)}u_r^2d\mathcal{H}^{n-1}-\frac{1}{r}\left(\frac{1-4\log r}{\log r(1-2\log r)}\right)\alpha(r)I(r)\\
			&-\frac{2\alpha(r)}{(n+2)r}\int_{\partial B_1}\frac{\mathcal{C}_r}{1-2\log r}d\mathcal{H}^{n-1}-\overline{Q}(r;u,x^0)\\
			=&\frac{n+2}{r}\left(M(r;\mathcal{C}_r)-M(r;u_r)\right)+\frac{\alpha(r)}{2r}\int_{\partial B_1}\left(\nabla u_r\cdot x -\frac{2}{\alpha(r)}u_r\right)^2d\mathcal{H}^{n-1}\\
			&+\frac{4\log r-1}{r\log r(2\log r-1)}\int_{ \partial B_1}u_r^2 d\mathcal{H}^{n-1}-\frac{1-4\log r}{r\log r(1-2\log r)}\alpha(r)I(r)\\
			&-\frac{2\alpha(r)}{(n+2)r}\int_{\partial B_1}\frac{\mathcal{C}_r}{1-2\log r}d\mathcal{H}^{n-1}-\overline{Q}(r;u,x^0)\\
			=&\frac{n+2}{r}\left(M(r;\mathcal{C}_r)-M(r;u_r)\right)+\frac{\alpha(r)}{2r}\int_{\partial B_1}\left(\nabla u_r\cdot x -\frac{2}{\alpha(r)}u_r\right)^2d\mathcal{H}^{n-1}\\
			&+\frac{1-4\log r}{r(-\log r)(1-2\log r)}\left(\alpha(r)I(r)-\int_{ \partial B_1}u_r^2 d\mathcal{H}^{n-1}  \right)\\
			&-\frac{2\alpha(r)}{(n+2)r}\int_{\partial B_1}\frac{\mathcal{C}_r}{1-2\log r}d\mathcal{H}^{n-1}-\overline{Q}(r;u,x^0).
			\end{align*}
			Recalling that $W(r;u,x^0)=\alpha(r)I(r)-\int_{ \partial B_1}u_r^2 d\mathcal{H}^{n-1}$, it follows that
			\begin{align*}
				\mathscr{E}'(r)=&\frac{n+2}{r}\left(M(r;\mathcal{C}_r)-M(r;u_r)\right)+\frac{\alpha(r)}{2r}\int_{\partial B_1}\left(\nabla u_r\cdot x -\frac{2}{\alpha(r)}u_r\right)^2d\mathcal{H}^{n-1}\\
				&+\frac{1-4\log r}{r(-\log r)(1-2\log r)} W(r;u,x^0) -\frac{2\alpha(r)}{(n+2)r}\int_{\partial B_1}\frac{\mathcal{C}_r}{1-2\log r}d\mathcal{H}^{n-1}-\overline{Q}(r;u,x^0).
			\end{align*}
			
		We claim that \begin{align}\label{claim}
		\frac{1-4\log r}{r(-\log r)(1-2\log r)} W(r;u,x^0) -\frac{2\alpha(r)}{(n+2)r}\int_{\partial B_1}\frac{\mathcal{C}_r}{1-2\log r}d\mathcal{H}^{n-1}-\overline{Q}(r;u,x^0)\geq 0,
		\end{align}
		for small $r$. In fact, 
		\begin{align}\label{Psi}
				\Psi(r;u,x^0):=&\frac{1-4\log r}{r(-\log r)(1-2\log r)} W(r;u,x^0) -\frac{2\alpha(r)}{(n+2)r}\int_{\partial B_1}\frac{\mathcal{C}_r}{1-2\log r}d\mathcal{H}^{n-1}-\overline{Q}(r;u,x^0)\nonumber\\
			=&\frac{1-4\log r}{r(-\log r)(1-2\log r)}\left(W(r;u,x^0)-\frac{2\alpha(r)\log r }{(4\log r -1)(n+2)}\int_{\partial B_1} u_rd\mathcal{H}^{n-1}\right)\nonumber\\
			&-\overline{Q}(r;u,x^0).
		\end{align}
		Here we focus on $$\Phi(r;u,x^0):=W(r;u,x^0)-\frac{2\alpha(r)\log r }{(4\log r -1)(n+2)}\int_{\partial B_1} u_rd\mathcal{H}^{n-1}\to 0, \qquad\text{as}\quad r\to 0+.$$
In fact, on the one hand, recalling that Lemma \ref{property} tells us the limit $\lim_{r\to 0+}W(r;u,x^0)$ exists and equals $\int_{ B_1}\frac{1}{2} u_0 dx$, thus we can obtain the limit
   
   $$\displaystyle\lim_{r\to 0+}\Phi(r;u,x^0)=\int_{ B_1}\frac{1}{2} u_0(x) dx-\frac{1}{2(n+2)}\int_{\partial B_1}u_0(x)d\mathcal{H}^{n-1}.$$
  On the other hand, the 2-homogeneity of $u_0$ implies that $$\displaystyle\frac{1}{n+2}\int_{\partial B_1}u_0(x)d\mathcal{H}^{n-1}=\int_{ B_1} u_0 (x)dx.$$
  Therefore,  it follows that $$\displaystyle\lim_{r\to 0+}\Phi(r;u,x^0)=0.$$
   Furthermore, we compute that
   	\begin{align}	\label{eq4.8}
   		\frac{d}{dr} \Phi(r;u,x^0)=&\frac{\alpha (r)}{r}\int_{\partial B_1} \left(\nabla u_r\cdot x -\frac{2}{\alpha(r)}u_r\right)^2d\mathcal{H}^{n-1}-\frac{2}{(n+2)r(4\log r -1)^2}\int_{\partial B_1} u_rd\mathcal{H}^{n-1}\nonumber\\
   		&-\frac{2\log r -1}{(n+2)(4\log r -1)}\int_{\partial B_1}\frac{1}{r}\left(\nabla u_r\cdot x-\frac{2}{\alpha(r)}u_r\right)d\mathcal{H}^{n-1}+Q(r;u,x^0)\nonumber\\
   		=&\frac{\alpha (r)}{r}\int_{\partial B_1} \left(\nabla u_r\cdot x -\frac{2}{\alpha(r)}u_r\right)^2d\mathcal{H}^{n-1}+P(r;u.x^0),
   	\end{align}
   where 
   \begin{align*}
   	P(r;u.x^0):=&Q(r;u,x^0)-\frac{2}{(n+2)r(4\log r -1)^2}\int_{\partial B_1} u_rd\mathcal{H}^{n-1}\\
   	&-\frac{2\log r -1}{(n+2)(4\log r -1)}\int_{\partial B_1}\frac{1}{r}\left(\nabla u_r\cdot x-\frac{2}{\alpha(r)}u_r\right)d\mathcal{H}^{n-1}.
   \end{align*}
Noting that $P(r;u,x^0)$ is an integrable function. Indeed, the foundation \eqref{order-q} indicates that for each $\delta > 0$,
 \begin{align*}
 	\left|\int_{0}^{r}\frac{1}{r}\int_{\partial B_1}\nabla u_r\cdot x-\frac{2}{\alpha(r)}u_rd\mathcal{H}^{n-1}dr\right|\leq\delta,
 \end{align*}
 provided that we first choose $r>0$ for fixed $x^0$ and then $|x-x^0|$ small enough.
	Then, one can see that
			\begin{equation*}
			\left(\Phi(r;u,x^0)-\int_{0}^{r}P(s;u,x^0)ds\right)'\geq 0,
		\end{equation*}
		and then the monotonicity implies that 
		\begin{equation*}
			\Phi(r;u,x^0)-\int_{0}^{r}P(s;u,x^0)ds\geq 	\displaystyle\lim_{r\to 0+}\Phi(r;u,x^0)=0. 
		\end{equation*}
It's nothing but
		\begin{equation*}
				\Phi(r;u,x^0)\geq \int_{0}^{r}P(s;u,x^0)ds.
		\end{equation*}
		Inserting it into \eqref{Psi}, we derive that
		\begin{equation*}
			\begin{aligned}
				\Psi(r;u,x^0)\geq \frac{1}{r}\frac{1-4\log r}{(-\log r)(1-2\log r)}\int_{0}^{r}P(s;u,x^0)ds-\overline{Q}(r;u,x^0).
			\end{aligned}
		\end{equation*}
	This is also a crucial aspect of our approach, focusing on the leading term, which is the term with the lowest order. Observe that $-\overline{Q}(r;u,x^0)\geq 0$  for sufficiently small $r$. On the other hand,  $\overline{Q}(r;u,x^0)$ is the leading term of the expression
		$$\frac{1}{r}\frac{1-4\log r}{(-\log r)(1-2\log r)}\int_{0}^{r}P(s;u,x^0)ds-\overline{Q}(r;u,x^0).$$
		It follows that $$\displaystyle\frac{1}{r}\frac{1-4\log r}{(-\log r)(1-2\log r)}\int_{0}^{r}P(s;u,x^0)ds-\overline{Q}(r;u,x^0)\geq 0,$$
		for sufficiently small $r$.  Therefore, the claim \eqref{claim} holds.

		Next, it together with the epiperimetric inequality implies the following linear differential inequality,
		\begin{align*}
			\frac{d}{dr}\mathscr{E}(r)\geq \frac{c}{r}\mathscr{E}(r) \quad\text{for any} \quad r\in(0,r_0),
		\end{align*}
		which gives a H\"older decay of $\mathscr{E}(r)$. Indeed, it can be concluded from the above claim \eqref{claim} that
				\begin{align}\label{eq4.9}
			\mathscr{E}'(r)\geq&\frac{n+2}{r}\left(M(r;\mathcal{C}_r)-M(r;u_r)\right)+\frac{\alpha(r)}{r}\int_{\partial B_1}\left(\nabla u_r\cdot x -\frac{2}{\alpha(r)}u_r\right)^2d\mathcal{H}^{n-1}\nonumber\\
			\geq&\frac{n+2}{r}\left(M(r;\mathcal{C}_r)-M(r;u_r)\right).	
			\end{align}

		 Furthermore, 
		recalling the epiperimetric inequality $$M(r;\mathcal{C}_r)\geq \frac{M(r;v)-\eta  W(0^+;u,x^0) -\eta \int_{0}^{r}\overline{Q}(s;u,x^0)ds}{1-\eta},$$ for any $r\in(0,r_0)$ and some ${v}\in W^{1,2}(B_1)$ with ${v}=\mathcal{C}_r$ on $\partial B_1$. Therefore, incorporating this result into \eqref{eq4.9}, we demonstrate that
 
	\begin{align}\label{eq4.10}
			\mathscr{E}'(r)
			\geq&\frac{n+2}{r}\left(\frac{M(r;u_r)-\eta  W(0^+;u,x^0) -\eta \int_{0}^{r}\overline{Q}(s;u,x^0)ds}{1-\eta}-M(r;u_r)\right)\nonumber\\
			=&\frac{(n+2)\eta}{(1-\eta)r}\left(\frac{M(r;u_r)-\eta  W(0^+;u,x^0) -\eta \int_{0}^{r}\overline{Q}(s;u,x^0)ds}{1-\eta}-M(r;u_r)\right)\nonumber\\
			=&\frac{(n+2)\eta}{1-\eta}\frac{\mathscr{E}(r)}{r}, \quad\text{for any} \quad r\in(0,r_0).
		\end{align}
		Here, we have used the fact that ${u}$ is the minimizer of the problem	\eqref{eq1.0} with $u_r =\mathcal{C}_r$ on $\partial B_1$. According to \eqref{eq4.10} and taking $\delta=\frac{\eta}{1-\eta}(n+2)$, one gets
		\begin{equation*}
			(r^{-\delta}\mathscr{E}(r))'\geq 0,  \quad\text{for any} \quad r\in(0,r_0),
		\end{equation*}
		 then, integrating the inequality from $r$ to $r_0$, from which we get that
		 \begin{equation}\label{eq4.11}
	\mathscr{E}(r)\leq \mathscr{E}(r_0)\left(\frac{r}{r_0}\right)^{\delta}, \quad\text{for any} \quad r\in(0,r_0).
	 \end{equation}
		
		Thanks to the decay of energy $\mathscr{E}(r)$ at hand, we can further obtain uniqueness of blow-up limit. Firstly, for $0<\rho<\sigma\leq r_0$,
\begin{align*}
	&\int_{\partial B_1} |u_{\sigma}-u_{\rho}|d\mathcal{H}^{n-1}\\
	\leq&\int_{\partial B_1}\int_\rho^\sigma\left|\frac{d u_r}{dr}\right|drd\mathcal{H}^{n-1}\\
	=&\int_\rho^\sigma \frac{ 1}{r^{n+1}(1-2\log r)}\int_{\partial B_r(x^0)}\left|\nabla u\cdot\nu-\frac{2}{\alpha(r)}\frac{ {u}}{r}\right|d\mathcal{H}^{n-1} dr.
\end{align*}
		By observing that the right-hand side of the above inequality resembles the square term in the monotonicity formula, we thus obtain the part we need through H\"older's inequality and matching, 
		\begin{align*}
			 	&\int_{\partial B_1} |u_{\sigma}-u_{\rho}|d\mathcal{H}^{n-1}\\
			 	\leq& \int_\rho^\sigma r^{-n-1}(1-2\log r)^{-1} \left(\int_{ \partial B_1} 1^2 d\mathcal{H}^{n-1}\right)^{\frac{1}{2}}\left(\int_{\partial B_r(x^0)}\left|\nabla u\cdot\nu-\frac{2}{\alpha(r)}\frac{ {u}}{r}\right|^2d\mathcal{H}^{n-1} \right)^{\frac{1}{2}}\\
			 	\leq&\sqrt{n\omega_n}\int_\rho^\sigma r^{-n-1}(1-2\log r)^{-1}r^{\frac{n-1}{2}}\left(\frac{\alpha(r)}{r^{n+2}(1-2\log r)^2}\right)^{\frac{1}{2}}\left(K(r;u,x^0)\right)^{\frac{1}{2}}dr,
		\end{align*}
		where $$K(r;u,x^0)=\displaystyle\frac{\alpha(r)}{r^{n+2}(1-2\log r)}\int_{\partial B_r(x^0)}\left|\nabla u\cdot\nu-\frac{2}{\alpha(r)}\frac{ {u}}{r}\right|^2d\mathcal{H}^{n-1}.$$
		Substituting this into the monotonicity formula, we obtain that

			\begin{align}\label{eq4.12}
				&\int_{\partial B_1} |u_{\sigma}-u_{\rho}|d\mathcal{H}^{n-1}\nonumber\\
				\leq&\sqrt{n\omega_n}\int_\rho^\sigma r^{-\frac{1}{2}}\sqrt{W'(r;u,x^0)-Q(r;u,x^0)}dr\nonumber\\
				\leq&\sqrt{n\omega_n}\left(\log \sigma-\log \rho\right)^{\frac{1}{2}}\left(W(\sigma;u,x^0)-W(\rho;u,x^0)-\int_{\rho}^{\sigma}Q(r;u,x^0)dr\right)^{\frac{1}{2}}\nonumber\\
				=&\sqrt{n\omega_n}\left(\log \sigma-\log \rho\right)^{\frac{1}{2}}\left(W(\sigma;u,x^0)-W(\rho;u,x^0)-\int_{\rho}^{\sigma} \overline{Q}(r;u,x^0)dr+\int_{\rho}^{\sigma}\displaystyle\frac{1}{r(1-2\log r)^{1+\gamma}}dr\right)^{\frac{1}{2}}\nonumber\\
				\leq&\sqrt{n\omega_n}\left(\log \sigma-\log \rho\right)^{\frac{1}{2}}\left(\overline{W}(\sigma;u,x^0)-\overline{W}(\rho;u,x^0)\right)^{\frac{1}{2}}\nonumber\\
				=&\sqrt{n\omega_n}\left(\log \sigma-\log \rho\right)^{\frac{1}{2}}\left(\mathscr{E}(\sigma)-\mathscr{E}(\rho)\right)^{\frac{1}{2}},
			\end{align}
				where $\gamma\in(0,1)$ as mentioned in \eqref{Q}.	
		
		 For any $0<2\rho<2r\leq r_0\ll 1$, there exist two positive integers $l<m$ such that $\rho\in [2^{-m-1}, 2^{-m})$ and $r \in [2^{-l-1}, 2^{-l})$. Hence recalling that \eqref{eq4.11} and \eqref{eq4.12}, we obtain
	
			\begin{align*}
				&\int_{\partial B_1}\left|{\frac{u(x^0+r x)}{r^2(1-2\log r)}-\frac{u(x^0+\rho x)}{\rho^2(1-2\log \rho)}}\right|d\mathcal{H}^{n-1}\\
				\leq& \sum_{j=l}^m \int_{\partial B_1}\int_{2^{-j-1}}^{2^{-j}}\left|\frac{d u_r}{dr}\right|drd\mathcal{H}^{n-1}\\
				\leq&\sqrt{n\omega_n}  \sum_{j=l}^m\left(\log (2^{-j})-\log(2^{-j-1})\right)^{\frac{1}{2}} \left(\mathscr{E}(2^{-j})-\mathscr{E}(2^{-j-1})\right)^{\frac{1}{2}}\\
				\leq&\sqrt{\log 2 n \omega_n}\sum_{j=l}^m\left(\mathscr{E}(2^{-j})\right)^{\frac{1}{2}}\\
				\leq&C_1(n)\sum_{j=l}^m\left(\mathscr{E}(r_0)\left(\frac{2^{-j}}{r_0}\right)^{\delta}\right)^{\frac{1}{2}}\\
				=&C_1(n)\left|\overline{W}(r_0;u,x^0)-\overline{W}(0+;u,x^0)\right|^{\frac{1}{2}}\sum_{j=l}^{+\infty}\left(r_0 2^j\right)^{-\frac{\delta}{2}}\\
				=&C_1(n)\left|\overline{W}(r_0;u,x^0)-\overline{W}(0+;u,x^0)\right|^{\frac{1}{2}} r_0^{-\frac{\delta}{2}}\frac{c^l}{1-c},
			\end{align*}
		where $c=2^{-\frac{\delta}{2}}\in (0,1)$. Therefore, the assumption $r \in [2^{-l-1}, 2^{-l})$ implies that $\displaystyle\frac{c^l}{1-c}<\displaystyle\frac{1}{1-c}r^{\frac{\delta}{2}}$. Furthermore, one can derive that
			\begin{equation*}
				\int_{\partial B_1}\left|{\frac{u(x^0+r x)}{r^2(1-2\log r)}-\frac{u(x^0+\rho x)}{\rho^2(1-2\log \rho)}}\right|d\mathcal{H}^{n-1}\leq 
		  C(n)\left|\overline{W}(r_0;u,x^0)-\overline{W}(0+;u,x^0)\right|^{\frac{1}{2}} \left(\frac{r}{r_0}\right)^{\frac{\delta}{2}},
	  \end{equation*}
	  where $\delta=\frac{(n+2)\eta}{1-\eta}$.
		 Consequently, let $\displaystyle\frac{{u}(x^0+\rho_j x)}{\rho_j^2(1-2\log \rho_j)}\to {u}_0$ as a certain sequence $\rho_j \rightarrow0+$, one can show \eqref{eq4.4}. Thus  we can conclude the proof.
		\qed

\section{Proof of Main Theorem \ref{regularity} of regularity of free boundary}
		In this section, we will present the proof of the main theorem (Theorem \ref{regularity}), which asserts that the free boundary of an open neighborhood around regular free boundary points $\mathcal{R}_{u}$ forms a $C^{1,\beta}$-surface in $\Omega$. For this purpose, we firstly need to verify the assumption (that is, the conditions for the epiperimetric inequality hold) of the blow-up limits decay rate estimate (Proposition \ref{uniqueness}) uniformly in an open neighborhood of a regular free boundary point. 
		
		\begin{lemma}\label{regularity1}
			Let $\mathcal{O}_{\mathbb{H}}$ be a compact set of points $x^0 \in \mathscr{F}(u)$ with the following property, at least one blow-up limit $u_0\in\mathbb{H}$ of $u$ at $x^0$, that is, $u_0(x)=\frac{1}{2}\max(x\cdot\nu(x^0),0)^2\in \mathbb{H}$ for some  $\nu(x^0)\in \partial B_1\subset \mathbb{R}^n$. Then there exist $r_0>0$ and $C<\infty$ such that
			$$
			\int_{\partial B_1}\left|\frac{u(x^0+rx)}{r^2(1-2\log r)}-\frac{1}{2}\max(x\cdot\nu(x^0),0)^2\right| d \mathcal{H}^{n-1}\leq Cr^{\frac{\delta}{2}},
			$$
			for every $x^0\in \mathcal{O}_{\mathbb{H}}$ and every $r\leq r_0<1$, where $\delta=\frac{(n+2)\eta}{1-\eta}$.
		\end{lemma}
		
		\begin{remark}
			Note that if we have the uniqueness of blow-up limit for the $x^0\in \mathcal{O}_{\mathbb{H}}$ mean that for any compact set $\mathcal{O}_{\mathbb{H}}\subset \mathcal{R}_u$.
		\end{remark}
			
		\begin{proof}
		Initially, we establish the non-emptiness of the compact set $\mathcal{O}_{\mathbb{H}}$. Given that the set $\mathscr{F}(u)$ is known to be non-empty, as defined in Proposition \ref{regular} and supported by Theorem \ref{isolated}. By leveraging the regularity of the solution $u$, it follows that $\mathcal{O}_{\mathbb{H}}$ is a compact set (detailed proof content can be found in \cite[Lemma 7.1]{dz}). Therefore, we conclusively establish that $\mathcal{O}_{\mathbb{H}}$ is not empty.
			The subsequent proof is then divided into four parts.
			
			(i) Firstly, for any $x\in \mathcal{O}_{\mathbb{H}}$, it is easy to see that for any $\epsilon>0$, there exists a $r_0(x, \epsilon)>0$ such that $$\overline{W}(r;u,x)\leq \epsilon +\frac{\omega_n}{2}\quad \text{for any}\quad r\in(0,r_0(x, \epsilon)).$$ Since $\overline{W}(r;u,x)$ is  increasing with respect to $r$, then using Dini's theorem  there exists  a uniform $r_0$ independent of the choice of $x$ such that $$\overline{W}(r;u,x)\leq \epsilon +\frac{\omega_n}{2}\ \ \text{for any $r\in(0,r_0)$ and any $x\in \mathcal{O}_{\mathbb{H}}$.}$$
			
			(ii) Secondly, if $\rho_j\rightarrow 0$, $x^j \in \mathcal{O}_{\mathbb{H}}$ and $u_j:=u(x^j+\rho_jx)/(\rho_j^2(1-2\log \rho_j))\rightarrow v$ in $W^{1,2}_{loc}(\mathbb{R}^n)$ as $j\rightarrow \infty$, then  recalling that the equality \eqref{eq2.5} ($\Delta u_0=\chi_{\{u_0>0\}}$) again, $v$ is a homogeneous global solution of degree 2 to the problem \eqref{eq1.1} and
			\begin{equation*}
				\begin{aligned}
					M_0(v)=&\displaystyle\lim_{j\rightarrow \infty}   \alpha(\rho_j)\int_{B_1}\frac{1}{2}|\nabla u_j(y)|^2+F(\rho_j,u_j(y))dy-\int_{\partial B_1}|u_j(y)|^2d\mathcal{H}^{n-1}\\
					=&\displaystyle\lim_{j\rightarrow \infty} W(\rho_j;u,x^j) =\displaystyle\lim_{j\rightarrow \infty}\overline{ W}(\rho_j;u,x^j) = \frac{\omega_n}{2}.
				\end{aligned}
			\end{equation*}
			 Due to the characterization of energy density Theorem \ref{isolated}, $v\in \mathbb{H}.$
			
			(iii) We claim that for any $\rho>0$ small enough, $\displaystyle\frac{u(\bar{x}+\rho x)}{\rho^2(1-2 \log \rho)}$ is uniformly close to $\mathbb{H}$ in the $W^{1,2}_{loc}(\mathbb{R}^n)$-topology for any $\bar{x}\in \mathcal{O}_{\mathbb{H}}$.
			
			To verify this claim, one may use the argument by contradiction. Assume this claim fails, then there exist  $\rho_j\rightarrow 0$ and $x^j\in \mathcal{O}_{\mathbb{H}}$ such that for any $h_{\nu}\in \mathbb{H}$, there holds
			\begin{equation}\label{eq5.1}
				\left\|\frac{u(x^j+\rho_j x)}{\rho_j^2 (1-2\log \rho_j)}-h_{\nu}\right\|_{W^{1,2}_{B_1}(\mathbb{R}^n)}\geq \delta >0.
			\end{equation} Owing to the growth estimate \cite[Lemma3.10, Theorem 3.11]{qs17}, we obtain that $|u_j|+|\nabla u_j|\leq C$, then $\|u_j\|_{L^{\infty}_{loc}(\mathbb{R}^n)}\leq C$. Thus
			$\|u_j\|_{W^{1,2}_{loc}(\mathbb{R}^n)}\leq C$. Hence, there is a convergent subsequence, still denoted by $u_j$,  such that $u_j\rightarrow u_0$ in $W^{1,2}_{loc}(\mathbb{R}^n)$.  Thanks  to the statement (ii), we conclude that $w\in \mathbb{H}$ ,  which contradicts with the fact \eqref{eq5.1}.
			
			(iv) Finally, we conclude the proof of Lemma \ref{regularity1}.
			
			In fact, the claim of (iii) implies that  the assumptions  in Proposition \ref{uniqueness} hold true and thus there is a constant  $C=C(n,r_0,\eta)>0$ such that
			\begin{equation*}
				\begin{aligned}
						\int_{\partial B_1}\left|\frac{u(x^0+rx)}{r^2(1-2\log r)}-u_0(x)\right| d\mathcal{H}^{n-1}&\leq C(n) \left|\overline{W}(r_0;u,x^0)-\overline{W}(0+;u,x^0) \right|^{\frac{1}{2}}\left(\frac{r}{r_0}\right)^{\frac{(n+2)\eta}{2(1-\eta)}}\\
						&\leq C(n)\epsilon^{\frac{1}{2}}
						\left(\frac{r}{r_0}\right)^{\frac{(n+2)\eta}{2(1-\eta)}}\\
					&\leq C(n,r_0,\eta)r^{\frac{\delta}{2}},
				\end{aligned}
			\end{equation*}
			for any ${x}^0\in \mathcal{O}_{\mathbb{H}}$ and $r\in(0, r_0)$. Here $u_0$ is the unique blow-up limit of $u$ at $x^0$ and $u_0\in \mathbb{H}$.
		\end{proof}

		\vspace{5pt}
		In the following, we use the Lemma \ref{regularity1} to prove the main result.
		
		\noindent{\it Proof of Theorem \ref{regularity}.}
		Consider $x^0\in \mathcal{R}_{u}\subset \mathscr{F}(u)$, by Lemma \ref{regularity1}, there exists a $\delta_0>0$ such that $B_{2\delta_0}(x^0)\subset \Omega$ and
		\begin{equation}\label{eq5.2}
			\int_{\partial B_1}\left|\frac{u(y+rx)}{r^2(1-2\log r)}-\frac{1}{2}\max(x\cdot\nu(y),0)^2 \right| d\mathcal{H}_x^{n-1}\leq C(n,r_0,\eta)r^{\frac{\delta}{2}},
		\end{equation}
		for every $y\in \mathcal{R}_{u}\cap \overline{B_{\delta_0}(x^0)} $ and every $r\leq\min(\delta_0,r_0)\leq r_0<1$.
		
		In the following, we split the proof into the establishment of several claims.
		
		\textbf{Claim 1}. $y\longmapsto\nu(y)$ is H\"older-continuous with exponent $\beta$ on $\mathcal{R}_u\cap \overline{B_{\delta_1}(x^0)}$ for some $\delta_1\in(0,\delta_0)$.
		
		\noindent{\it Proof of Claim 1}. In fact, for any $y, z \in \mathcal{R}_u\cap \overline{B_{\delta_1}(x^0)}$, the Cauchy inequality leads to that
	 
			\begin{align*}
				&\int_{\partial B_1}\left|\frac{1}{2}\max(x\cdot\nu(y),0)^2-\frac{1}{2}\max(x\cdot\nu(z),0)^2\right|d\mathcal{H}_x^{n-1}\\
				=&\int_{\partial B_1}\Bigg{|}\left(\frac{1}{2}\max(x\cdot\nu(y),0)^2-\frac{u(y+rx)}{r^2(1-2\log r)}\right)+\left(\frac{u(y+rx)}{r^2(1-2\log r)}-\frac{u(z+rx)}{r^2(1-2\log r)} \right) \\
				&+\left(\frac{u(z+rx)}{r^2(1-2\log r)}  -\frac{1}{2}\max(x\cdot\nu(z),0)^2\right)\Bigg{|}d\mathcal{H}_x^{n-1}\\
				\leq& 2Cr^{\frac{\delta}{2}} +\int_{\partial B_1}\left|\frac{u(y+rx)}{r^2(1-2\log r)}-\frac{u(z+rx)}{r^2(1-2\log r)} \right|d\mathcal{H}_x^{n-1}\\
				\leq&2Cr^{\frac{\delta}{2}} +\int_{\partial B_1}\int_{0}^{1}\left| \frac{\nabla u \left((y+rx +t(y-z)\right)}{r^2(1-2\log r)}\right||y-z|dtd\mathcal{H}_x^{n-1}.
			\end{align*}
				
			Based on this, we revisit the optimal Log-Lipschitz regularity for the gradient of a minimizer of \eqref{eq1.0} established by de Queiroz and Shahgholian in \cite[Theorem 3.11]{qs17}, namely, let $u$ be a minimizer of \eqref{eq1.0} and $\Omega'\subset\subset \Omega$, there exist $r_0>0$ and a constant $C>0$ depending both only on $\text{dist}(\Omega',\partial \Omega)$, $\varphi$ and $n$ such that, if $x\in \Omega'$ with $d(x)=\text{dist}(x,\partial \{u>0\})\leq r_0$, then
			 \begin{equation*}
			 	|\nabla u(x)|\leq Cd(x)\log \frac{1}{d(x)}.
			 \end{equation*}
				Therefore, for small enough $r_0$,  we can derive that 
						
				\begin{align*}
						&\int_{\partial B_1}\left|\frac{1}{2}\max(x\cdot\nu(y),0)^2-\frac{1}{2}\max(x\cdot\nu(z),0)^2\right|d\mathcal{H}_x^{n-1}\\
						\leq&2Cr^{\frac{\delta}{2}}\int_{\partial B_1}\int_{0}^{1} \frac{Cd(x)\log \frac{1}{d(x)}}{r^2(1-2\log r)}|y-z|dtd\mathcal{H}_x^{n-1}\\
						\leq&2Cr^{\frac{\delta}{2}}\int_{\partial B_1}\int_{0}^{1} \frac{C\max(r;|y-z|)\log \frac{1}{\max(r;|y-z|)}}{r^2(1-2\log r)}|y-z|dtd\mathcal{H}_x^{n-1}\\
				\leq&2Cr^{\frac{\delta}{2}}+C_1r^{-1}|y-z|\\
				=&2Cr^{\frac{\delta}{2}}+C_2|y-z|^{1-\gamma}\\
				\leq&(2C+C_2)|y-z|^{\gamma {\frac{\delta}{2}}},
			\end{align*}
		where we choose $0<\gamma:=\left(1+\frac{\delta}{2}\right)^{-1}<1$ and $r:=|y-z|^{\gamma}\leq \min (\delta_0,r_0)\leq r_0$. Indeed, let
		\begin{equation*}
			\gamma \frac{\delta}{2}=1-\gamma,
		\end{equation*}
		that is $\gamma=\left(1+\frac{\delta}{2}\right)^{-1}$.
		For convenience, we denote $\beta:=	\gamma \frac{\delta}{2}=1-\gamma=\frac{\delta}{2+\delta}$, where $\delta=\frac{(n+2)\eta}{1-\eta}>0$.
		
	   On the other hand, through an indirect argument, it can be shown that the left side of the above inequality is satisfied
			\begin{align*}
				\int_{\partial B_1}\left|\frac{1}{2}\max(x\cdot\nu(y),0)^2-\frac{1}{2}\max(x\cdot\nu(z),0)^2\right|d\mathcal{H}^{n-1}\nonumber
				\geq c(n)\left(\left|\nu(y)-\nu(z)\right|\right).
			\end{align*}
 
		Therefore, we finish the proof of Claim 1.
		
		\textbf{Claim 2}. $\forall \ \epsilon>0$, there exists a $\delta_2\in (0,\delta_1)$ such that for any  $z\in \mathcal{R}_{u}\cap \overline{B_{\delta_0}(x^0)}$,  we obtain
		
		\begin{equation*}
			\begin{split}
				|u(y)|&=0 \quad \text{for}\quad y\in  \overline{B_{\delta_2}(y)} \quad \text{satisfying}\quad (y-z)\cdot \nu(z)<-\epsilon|y-z|, \\
				\text{and}\qquad\qquad&\\
				|u(y)|&>0 \quad \text{for}\quad y\in  \overline{B_{\delta_2}(y)} \quad \text{satisfying}\quad (y-z)\cdot \nu(z)>\epsilon|y-z|.
			\end{split}
		\end{equation*}
		
		\noindent{\it Proof of Claim 2}. Based on the inequality \eqref{eq5.2}, we can draw a conclusion using the method of contradiction.
		
		\textbf{Claim 3}. There exists a $\delta_3\in(0,\delta_2)$ such that $\partial \{ |u|>0\}$ is in $B_{\delta_3}(x^0)$ the graph of a differentiable function.
		
		\noindent{\it Proof of Claim 3}. Based on Claim 1 and Claim 2, we can directly prove this assertion. Notice that from Claim 1, it follows that there exists a uniform constant $C>0$ such that $|\nu(x)-\nu(y)|\leq  C|x-y|^{\beta},\quad \beta:=\frac{\delta}{2+\delta}\in(0,1),\quad\delta=\frac{(n+2)\eta}{1-\eta}>0$ for any $x, y\in  \overline{B_{\delta_3}(x^0)}$, which concludes $\partial \{ u>0\}$ is the graph of a  $C^{1, \beta}(\overline{B'_{\delta_3}(0)})$ function on $\overline{B_{\delta_3}(x^0)}$.

	\end{document}